\documentclass{amsart}

\usepackage{amssymb}
\usepackage{amsthm}
\usepackage{amsmath}
\usepackage{graphicx}
\usepackage{pifont}
\usepackage{txfonts}

\theoremstyle{plain}
\newtheorem{theorem}{Theorem}[section]
\newtheorem{proposition}[theorem]{Proposition}%[section]
\newtheorem{corollary}[theorem]{Corollary}%[section]
\newtheorem{lemma}[theorem]{Lemma}%[section]

\newtheorem*{Asymptotics}{Asymptotic behaviors for Seifert fibered spaces}
\newtheorem*{SurgeryFormula}{Surgery formula for the asymptotics}

\newtheorem*{conjecture*}{Conjecture}
\theoremstyle{definition}
\newtheorem{definition}[theorem]{Definition}
\newtheorem*{definition*}{Definition}%[section]
\newtheorem*{example*}{Example}
\newtheorem*{notation*}{Notation}
\newtheorem*{notation-conv*}{Notation and convention}
\newtheorem*{convention*}{Convention}
\theoremstyle{remark}
\newtheorem{remark}[theorem]{Remark}%[sectio]

\def\co{\colon\thinspace}

\newcommand{\Z}{\mathbb{Z}}

\newcommand{\R}{\mathbb{R}}
\newcommand{\C}{\mathbb{C}}
\newcommand{\F}{\mathbb{F}}

\newcommand{\SL}[1][2]{\mathrm{SL}_{#1}(\C)}

\newcommand{\SU}{\mathrm{SU}(2)}
\newcommand{\GL}{\mathrm{GL}}

\newcommand{\Hyp}{\mathrm{Hyp}}
\newcommand{\Para}{\mathrm{Para}}

\newcommand{\trace}{{\rm tr}\,}
\newcommand{\I}{I}
\newcommand{\bm}[1]{\mbox{\boldmath{$#1$}}}

\newcommand{\bnd}[1]{\partial_{#1}}
\newcommand{\basisM}[1][i]{\mathbf{c}^{#1}}
\newcommand{\basisBt}[1][i]{\tilde{\mathbf{b}}^{#1}}
\newcommand{\resultingMfd}{M\hbox{$\big(\frac{\alpha_1}{\beta_1}, \ldots, \frac{\alpha_m}{\beta_m}\big)$}}
\newcommand{\mfd}{M}
\newcommand{\Seifert}{M\hbox{$\big(\frac{1}{-b}, \frac{\alpha_1}{\beta_1}, \ldots, \frac{\alpha_m}{\beta_m}\big)$}}
\newcommand{\SeifertZHS}{M\hbox{$\big(\frac{\alpha_1}{\beta_1}, \ldots, \frac{\alpha_m}{\beta_m}\big)$}}
\newcommand{\ExI}{M\hbox{$\big(\frac{2}{1}, \frac{3}{-1}, \frac{7}{-1}\big)$}}
\newcommand{\ExII}{M\hbox{$\big(\frac{2}{1}, \frac{3}{-2}, \frac{5}{-2}, \frac{7}{4}\big)$}}
\newcommand{\ExIII}{M\hbox{$\big(\frac{5}{3}, \frac{6}{-1}, \frac{7}{-3}\big)$}}
\newcommand{\CharVar}[1]{\mathcal{R}(#1)}

\newcommand{\univcover}[1]{{\widetilde #1}}
\newcommand{\lift}[1]{\tilde{#1}}

\newcommand{\Sym}[1]{\mathop{\mathrm{Sym}^{#1}(\C^2)}}
\newcommand{\Tor}[2]{\mathop{\mathrm{Tor}}\nolimits (#1;#2)}
\newcommand{\ie}{i.e.,\,}
\begin{document}

%%%%%%%%%%%%%%%%%%%%%%%%%%%%%%%%%%%%%%%%%% 
% title 
%%%%%%%%%%%%%%%%%%%%%%%%%%%%%%%%%%%%%%%%%% 

\title[Surgery formula for asymptotics of R-torsion and Seifert fibered spaces]{
  A surgery formula for the asymptotics of the higher dimensional Reidemeister torsion
  and Seifert fibered spaces
}

%%%%%%%%%%%%%%%%%%%%%%%%%%%%%%%%%%%%%%%%%% 
% author names and addresses
%%%%%%%%%%%%%%%%%%%%%%%%%%%%%%%%%%%%%%%%%% 
\author{Yoshikazu Yamaguchi}

\address{Department of Mathematics,
  Akita University,
  1-1 Tegata-Gakuenmachi, Akita, 010-8502, Japan}
\email{shouji@math.akita-u.ac.jp}

%% \date{\today}

%%%%%%%%%%%%%%%%%%%%%%%%%%%%%%%%%%%%%%%%%% 
% Subject classification 
%%%%%%%%%%%%%%%%%%%%%%%%%%%%%%%%%%%%%%%%%% 

%%%%%%%%%%%%%%%%%%%%%%%%%%%%%%%%%%%%%%%%%% 
% abstract
%%%%%%%%%%%%%%%%%%%%%%%%%%%%%%%%%%%%%%%%%% 
\begin{abstract}
  We give a surgery formula for the asymptotic behavior of the sequence
  given by the logarithm of the higher dimensional Reidemeister torsion.
  Applying the resulting formula to Seifert fibered spaces,
  we show that the growth of the sequences has the same order as
  the indices and we give the explicit values for the limits of the leading coefficients.
  There are finitely many possibilities as the limits of the leading coefficients
  for a Seifert fibered space. %% which are assigned to each component in the representation spaces.
  We also show that the maximum is given by $-\chi \log 2$
  where $\chi$ is the Euler characteristic of the base orbifold
  for a Seifert fibered space.
  These limits of the leading coefficients give a locally constant function on
  a character variety. This function takes 
  the maximum $-\chi \log 2$ only on the top-dimensional components of the $\SU$-character varieties
  for Seifert fibered homology spheres.
\end{abstract}

%%%%%%%%%%%%%%%%%%%%%%%%%%%%%%%%%%%%%%%%%%%%%%%%%%%%%%%%%%%%%%%%%%%% 
% body of paper
%%%%%%%%%%%%%%%%%%%%%%%%%%%%%%%%%%%%%%%%%%%%%%%%%%%%%%%%%%%%%%%%%%%% 

\maketitle
%% \tableofcontents

\section{Introduction}
This paper is devoted to the study of the asymptotics of
the higher dimensional Reidemeister torsion for {\it Seifert fibered spaces}.
The higher dimensional Reidemeister torsion is defined for a $3$-manifold and 
a sequence of homomorphisms from the fundamental group into special linear groups.
This sequence is given by the composition of an {\it $\SL$-representation} of a fundamental group with
$n$-dimensional irreducible representations of $\SL$.
It is of interest to observe the asymptotic behavior of
the higher dimensional Reidemeister torsion on the index $n$.

This study is motivated by the works of
W.~M{\"u}ller~\cite{Muller:AsymptoticsAnalyticTorsion} and
P.~Menal-Ferrer and J.~Porti~\cite{FerrerPorti:HigherDimReidemeister},
which revealed the relationship between the asymptotic behaviors of
the Ray--Singer and the Reidemeister torsions for a hyperbolic $3$-manifold and
its hyperbolic volume.  Their invariants are defined by a
sequence of $\SL[n]$-representations induced from the holonomy
representation corresponding to the complete hyperbolic structure.
The sequences of the Ray--Singer and the Reidemeister torsions have exponential
growth and the logarithms have the order of $n^2$.  They showed that
the leading coefficient of the logarithm converges to the product of
the hyperbolic volume and $-1/(4\pi)$.

We can also consider the similar sequences given by  
the Reidemeister torsions for non--hyperbolic $3$-manifolds,
especially for Seifert fibered spaces.
From the results of M{\"u}ller
and Menal-Ferrer and Porti,
it is expected that 
the growth of the logarithm of the higher dimensional Reidemeister torsion
for a Seifert fibered space has the order of smaller than $n^2$
and the leading coefficient converges to some geometric quantity of the Seifert fibered space.
Actually we will see that the order of growth is the same as $n$ and 
the limits of the leading coefficients form a finite set,
in which the maximum is given by $-\chi \log 2$
where 
$\chi$ is the Euler characteristic of the base orbifold of a Seifert fibered space.

To observe the asymptotic behaviors for Seifert fibered spaces,
we establish a surgery formula
for gluing solid tori to a compact orientable $3$-manifold with torus boundary,
since every Seifert fibered space admits the canonical decomposition into the trivial $S^1$-bundle 
over a compact surface and several solid tori (for details, see Subsection~\ref{subsec:Seifert}).

There is a problem which we have to consider on the choice of homomorphisms from fundamental groups
into $\SL$
since Seifert fibered spaces do not admit complete hyperbolic structures.
Concerning this problem, 
we focus on the {\it acyclicity} properties for 
the induced twisted chain complexes defined by $\SL[n]$-representations
for hyperbolic $3$-manifolds.
It was shown in~\cite{Raghunathan65:CohomologyVanishing, FerrerPorti:TwistedCohomology}
that for a hyperbolic $3$-manifold, 
the induced $\SL[2N]$-representations from the holonomy representation define
{\it acyclic} chain complexes, 
\ie
all of those homology groups vanish.
We will restrict our attention to $\SL$-representations such that 
all the induced $\SL[2N]$-representations define the acyclic twisted chain complexes.

The following is  a surgery formula (Theorem~\ref{thm:surgery_formula}) for the limits of sequences given by 
the logarithms of the higher dimensional Reidemeister torsions
under our acyclicity conditions (see Definition~\ref{def:acyclicity_conditions} for the acyclicity conditions).
\begin{SurgeryFormula}[Theorem~\ref{thm:surgery_formula}]
  Let $\resultingMfd$ be a closed orientable $3$-manifold obtained by
  gluing solid tori $S_1, \ldots, S_m$
  with the slopes $\frac{\alpha_1}{\beta_1}, \ldots, \frac{\alpha_m}{\beta_m}$
  to a compact orientable manifold $M$ with torus boundary.

  We denote by $\rho$ an $\SL$-representation of $\pi_1(\mfd)$
  which can be extended to a homomorphism of $\pi_1(\resultingMfd)$
  and by $\rho_{2N}$ the $2N$-dimensional representation induced from $\rho$.
  Then the asymptotics of $\log|\Tor{\resultingMfd}{\rho_{2N}}|$ is expressed as follows:
  \begin{enumerate}
  \item
    $\displaystyle{
    \lim_{N \to \infty}
    \frac{
      \log|\Tor{\resultingMfd}{\rho_{2N}}|
    }{
      (2N)^2
    }
    =
    \lim_{N \to \infty}
    \frac{
      \log|\Tor{\mfd}{\rho_{2N}}|
    }{
      (2N)^2
    },
    }$
  \item
    $\displaystyle{
    \lim_{N \to \infty}
    \frac{
      \log|\Tor{\resultingMfd}{\rho_{2N}}|
    }{
      2N
    }
    =
    \lim_{N \to \infty}
    \frac{
      \log|\Tor{\mfd}{\rho_{2N}}|
    }{
      2N
    }
    -\log 2 \sum_{j=1}^m \frac{1}{\lambda_j}
    }$
  \end{enumerate}
  where $2\lambda_j$ is the order of $\rho(\ell_j)$ and
  $\ell_j$ is the homotopy class of the core in $S_j$.
\end{SurgeryFormula}

Applying the above surgery formula to the decomposition of Seifert fibered spaces,
we can show the asymptotic behaviors for Seifert fibered spaces.
\begin{Asymptotics}[Theorem~\ref{thm:asymptotics_Seifert} and Corollary~\ref{cor:maximal_leading_term}]
  Let $\Seifert$ denote a Seifert fibered space with the Seifert index:
  $$
  \{b, (o, g); (\alpha_1, \beta_1), \ldots, (\alpha_m, \beta_m)\}.
  $$
  Then we can express the asymptotic behavior of $\log |\Tor{\Seifert}{\rho_{2N}}|$ as follows:
  \begin{enumerate}
  \item
    $\displaystyle{
    \lim_{N \to \infty}
    \frac{
      \log |\Tor{\Seifert}{\rho_{2N}}|
    }{
      (2N)^2
    }
    =0
  }$,
  \item
    $\displaystyle{
    \lim_{N \to \infty}
    \frac{
      \log |\Tor{\Seifert}{\rho_{2N}}|
    }{
      2N
    }
    = - \Big(2-2g - \sum_{j=1}^m \frac{\lambda_j - 1}{\lambda_j}\Big) \log 2
  }$
  \end{enumerate}
  where $2\lambda_j$ is the order of $\rho(\ell_j)$ for the exceptional fiber $\ell_j$. 
  In particular, the second equality can be written as
  $$
  \lim_{N \to \infty}
  \frac{
    \log |\Tor{\Seifert}{\rho_{2N}}|
  }{
    2N
  }
  = -\Big(2-2g - \sum_{j=1}^m \frac{\alpha_j -1}{\alpha_j}\Big) \log 2
  - \Big(\sum_{j=1}^m\Big(\frac{1}{\lambda_j} - \frac{1}{\alpha_j}\Big) \Big)\log 2.
  $$
\end{Asymptotics}
The first term in the right hand side is equal to
$-\chi \log 2$ where $\chi$ is the Euler characteristic of the base orbifold.
Each $\lambda_j$ turns out to be a divisor of $\alpha_j$.
Hence we can see that the maximum of the above limit is given by 
$-\chi \log 2$.
Note that we only require that the original $\SL$-representation sends a regular fiber
to $-\I$ for a Seifert fibered space.

Moreover the limit of the leading coefficient is determined by each component in the representation spaces.
By the invariance of Reidemeister torsion under the conjugation of representations,
we can assign each component of the character variety to the limit of the leading coefficient,
namely, we can define a locally constant function on the character variety for 
a Seifert fibered space.
We will discuss on which components our locally constant function takes the maximum and minimum
for the $\SU$-character varieties of Seifert fibered homology $3$-spheres in detail.
We can find the components which give the maximum in top dimensional components of the $\SU$-character variety.
In particular, 
for Seifert fibered homology $3$-spheres given by sequences of prime integers,
Theorem~\ref{thm:max_min_leading_coeff} shows that 
all top-dimensional components give the maximum and some $0$-dimensional components give the minimum.

\subsection*{Organization}
We review the definition of the Reidemeister torsion and 
the construction of the higher dimensional ones
in Section~\ref{sec:preliminaries}.
Section~\ref{sec:surgery_formula} is devoted to establish our surgery formula
under the acyclicity conditions which are deduced from the observation in Subsection~\ref{subsec:acyc_condition}.
The examples of the surgery formula for integral surgeries along torus knots
are exhibited in Subsection~\ref{subsec:example_torusknots}.
We discuss the asymptotic behaviors of the sequences given by 
the logarithm of the higher dimensional Reidemeister torsion for Seifert fibered spaces
in Section~\ref{sec:asymptotics_Seifert}.
We review on Seifert fibered spaces and prepare notations in Subsection~\ref{subsec:Seifert}.
Subsection~\ref{subsec:asymptotics_Seifert} gives a general formula of the asymptotic
behavior for a Seifert fibered space. 
Furthermore 
we observe the relation between the limits of the leading coefficients and 
components in the $\SU$-character varieties for Seifert fibered homology $3$-spheres
in Subsection~\ref{subsec:leadingCoeff_Seifert}.
The last Subsection~\ref{subsec:examples} gives the explicit examples
of limits of the leading coefficients and the $\SU$-character varieties.

\section{Preliminaries}
\label{sec:preliminaries}
Although one can find the similar preliminaries in~\cite{yamaguchi:RtorTorusKnots},
we give a review on the Reidemeister torsion, needed in our observation, 
to make this article self-contained.
\subsection{Reidemeister torsion}
\label{section:Rtorsion}
\subsubsection*{Torsion for acyclic chain complexes}
{\it Torsion} is an invariant defined for based chain complexes.
We denote by $(C_*, \basisM[*] = \cup_i \basisM)$ a {\it based} chain complex:
$$
C_*: 0 \to C_n \xrightarrow{\bnd{n}}
C_{n-1} \xrightarrow{\bnd{n-1}}
\cdots \xrightarrow{\bnd{2}}
C_1 \xrightarrow{\bnd{1}}
C_0 \to 0
$$
where each chain module $C_i$ is a vector space over a field $\F$
and equipped with a basis $\basisM$. 
We are mainly interested in an {\it acyclic} chain complex $C_*$
which has the trivial homology group, \ie
$H_*(C_*) = \bm{0}$.
The chain complex $C_*$ also has a basis determined by the 
boundary operators $\bnd{i}$,
which arises from the following decomposition of chain modules.

We suppose that a based chain complex $(C_*, \basisM[*])$ is acyclic.
For each boundary operator $\bnd{i}$,
we denote $\ker \bnd{i} \subset C_i$ by $Z_i$ and
the image of $\bnd{i}$ by $B_i \subset C_{i-1}$.
The chain module $C_i$ is expressed as the direct sum of $Z_i$ and the lift of $B_i$,
denoted by $\lift{B}_i$.
Moreover we can rewrite the kernel $Z_i$ as
the image of boundary operator $\bnd{i+1}$:
\begin{align*}
  C_i
  &= Z_i \oplus \lift{B}_i \\
  &= \bnd{i+1}\lift{B}_{i+1} \oplus \lift{B}_i 
\end{align*}
where $Z_i = B_{i+1}$ is written as $\bnd{i+1}\lift{B}_{i+1}$.

We denote by $\basisBt$ a basis of $\lift{B}_{i}$.
Then the set $\bnd{i+1}(\basisBt[i+1]) \cup \basisBt$
forms a new basis of the vector space $C_i$.
We define the {\it torsion} of $(C_*, \basisM[*])$ as
the following alternating product of determinants of base change matrices:
\begin{equation}
  \label{eqn:def_torsion_complex}
  \mathrm{Tor}(C_*, \basisM[*])
  = \prod_{i \geq 0} 
  \left[
    \bnd{i+1}(\basisBt[i+1]) \cup \basisBt \,/\, \basisM
    \right]^{(-1)^{i+1}}
  \in \F^* = \F \setminus \{0\}
\end{equation}
where $[ \bnd{i+1}(\basisBt[i+1]) \cup \basisBt \,/\, \basisM ]$
denotes the determinant of the base change matrix from
the given basis $\basisM$ to 
the new one $\bnd{i+1}(\basisBt[i+1]) \cup \basisBt$.

Note that the right hand side is independent of the choice of bases $\basisBt$.
The alternating product in~\eqref{eqn:def_torsion_complex} is determined by the based chain complex
$(C_*, \basisM[*])$.

\subsubsection*{Reidemeister torsion for CW--complexes}
We apply the torsion~\eqref{eqn:def_torsion_complex} of a based chain complex to the 
{\it twisted chain complex} given by 
a CW--complex and a homomorphism from its fundamental group to some linear group. 
Let $W$ denote a finite CW--complex and $(V, \rho)$ a representation of $\pi_1(W)$,
which means $V$ is a vector space over $\F$ and $\rho$ is a homomorphism from $\pi_1(W)$ into $\GL(V)$.
We will call $\rho$ a $\GL(V)$-representation of $\pi_1(W)$ simply.
\begin{definition}
  \label{def:twistedcomplex}
  We define the twisted chain complex $C_*(W; V_\rho)$ which consists of the twisted chain module as:
  $$
  C_i (W; V_\rho) :=V \otimes_{\Z[\pi_1(W)]} C_i (\univcover{W};\Z)
  $$
  where $\univcover{W}$ is the universal cover of $W$ and $C_i(\univcover{W};\Z)$
  is a left $\Z[\pi_1(W)]$-module in which the action of $\pi_1(W)$ is given by the covering transformation.
  In taking the tensor product, we regard $V$ as a right $\Z[\pi_1(W)]$-module under the homomorphism $\rho^{-1}$.
  We identify a chain $\bm{v} \otimes \gamma c$ with $\rho(\gamma)^{-1}(\bm{v}) \otimes c$
  in $C_i (W; V_\rho)$.
\end{definition}
We call $C_*(W;V_\rho)$ the twisted chain complex with the coefficient $V_\rho$.
Choosing a basis of the vector space $V$, we give a basis of the twisted chain complex $C_*(W; V_\rho)$.
To be more precise, let $\{e^{i}_1, \ldots, e^{i}_{m_i}\}$ be the set of $i$-dimensional cells of $W$ and 
$\{\bm{v}_1, \ldots, \bm{v}_{d}\}$ a basis of $V$ where $d = \dim_{\F} V$. 
Choosing a lift $\lift{e}^{i}_j$ of each cell and taking tensor product with the basis of $V$,
we have the following basis of $C_i(W;V_\rho)$:
$$
\basisM(W;V)=
\{
\bm{v}_1 \otimes \lift{e}^{i}_1, \ldots, \bm{v}_d \otimes \lift{e}^{i}_1,
\ldots,
\bm{v}_1 \otimes \lift{e}^{i}_{m_i}, \ldots, \bm{v}_d \otimes \lift{e}^{i}_{m_i}
\}.
$$
We denote by $H_*(W;V_\rho)$
the homology group and call it {\it the twisted homology group} and 
say that $\rho$ is acyclic if the twisted homology group vanishes.
Regarding $C_*(W; V_\rho)$ as a based chain complex,
we define the Reidemeister torsion for $W$ and an acyclic representation $(V, \rho)$
as the torsion of $C_*(W;V_\rho)$, \ie
\begin{equation}
  \label{eqn:def_RtorsionCW}
  \mathrm{Tor}(W; V_\rho) = \mathrm{Tor}(C_*(W;V_\rho), \basisM[*](W;V))
\in \F^* 
\end{equation}
up to a factor in $\{\pm \det(\rho(\gamma)) \,|\, \gamma \in \pi_1(W)\}$
since we have many choices of lifts $\lift{e}^{i}_j$
and orders and orientations of cells $e^{i}_j$.
We call $\mathrm{Tor}(W; V_\rho)$ the Reidemeister torsion of $W$ and a $\GL(V)$-representation $\rho$.

\begin{remark}
  We mention some well--definedness of the torsion~\eqref{eqn:def_RtorsionCW}:
  \begin{itemize}
  \item
    The acyclicity of $C_*(W;V_\rho)$ implies that the Euler characteristic of $W$ is zero.
    Then the torsion~\eqref{eqn:def_RtorsionCW} of $C_*(W;V_\rho)$ is independent of the choice of
    a basis in $V$.
  \item
    If we choose an $\mathrm{SL}(V)$-representation $\rho$ with an even dimensional $V$,
    then the Reidemeister torsion $\mathrm{Tor}(W; V_\rho)$ has no indeterminacy.
  \item
    The Reidemeister torsion has an invariance under the conjugation of representations.
  \end{itemize}
\end{remark}

The following lemma for torsion~\eqref{eqn:def_torsion_complex}
will be needed to derive our surgery formula:
\begin{lemma}[Multiplicativity Lemma]
  \label{lemma:MultLemma}
  Let
  $0 \to (C'_*, \bar{\mathbf{c}}^{*}) \to (C_*, \basisM[*]) \to (C''_*, \bar{\bar{\mathbf{c}}}^{*}) \to 0$
  be 
  the short exact sequence of based chain complexes
  such that $[\bar{\mathbf{c}}^{i} \cup \bar{\bar{\mathbf{c}}}^{i} / \basisM] =1$ for all $i$.
  Suppose that any two of the complexes are acyclic. 
  Then the third one is also acyclic and the torsion 
  of the three complexes are well-defined. Furthermore we have the next equality:
  $$
  \mathrm{Tor}(C_*, \basisM[*])
  = (-1)^{\sum_{i \geq 0} \beta'_{i-1}\beta''_{i}}
  \mathrm{Tor}(C'_*, \bar{\mathbf{c}}^{*}) \mathrm{Tor}(C''_*, \bar{\bar{\mathbf{c}}}^{*})
  $$
  where $\beta'_i= \dim_\F \partial C'_{i+1}$ and $\beta''_i = \dim_\F \partial C''_{i+1}$.
\end{lemma}
We refer to Milnor's survey~\cite{Milnor:1966} and Turaev's
book~\cite{Turaev:2000} for more details on Reidemeister torsion.

\subsection{Higher dimensional Reidemeister torsion for $\SL$-representations}
We will consider a sequence of the Reidemeister torsion of a finite CW-complex $W$.
This sequence corresponds to the sequence of 
the $\SL[n]$-representations of $\pi_1(W)$,
induced by an $\SL$-representation.
Let $\rho$ be a homomorphism from $\pi_1(W)$ to $\SL$.
Then the pair $(\C^2, \rho)$ is an $\SL$-representation of $\pi_1(W)$
by the standard action of $\SL$ to $\C^2$.
It is known that the pair of the symmetric product $\Sym{n-1}$ and 
the induced action by $\SL$
gives an $n$-dimensional irreducible representation of $\SL$.
The symmetric product $\Sym{n-1}$
can be identified  with the vector space $V_n$ of 
homogeneous polynomials on $\C^2$ with degree $n-1$, 
\ie
$$
V_n =
\mathrm{span}_{\C}\langle
z_1^{n-1}, z_1^{n-2}z_2, \ldots, z_1^{n-k-1} z_2^k, \ldots,z_1 z_2^{n-2}, z_2^{n-1}
\rangle
$$
and the action of $A \in \SL$ is expressed as 
\begin{equation}
  \label{eqn:action_SL2}
  A \cdot p(z_1, z_2) = p (A^{-1} \left(\begin{smallmatrix} z_1 \\ z_2\end{smallmatrix} \right))
    \quad \text{for} \quad p(z_1, z_2) \in V_n.
\end{equation}
We write $(V_n, \sigma_n)$ for the representation given by the action~\eqref{eqn:action_SL2} of $\SL$
where $\sigma_n$ denotes the homomorphism from $\SL$ into $\GL (V_n)$.
In fact the image of $\sigma_n$ is contained in $\SL[n]$.
We will see this after the definition of the higher dimensional Reidemeister torsion.

We write $\rho_n$ for the composition $\sigma_n \circ \rho$.
Then we can take $V_n$ as a coefficient of a twisted chain complex for $W$
since the vector space $V_n$ is a right $\Z[\pi_1(W)]$-module of $\pi_1(W)$.
We denote by $C_*(W; V_n)$ this twisted chain complex of $W$ defined by $(V_n, \rho_n)$.
We will drop the subscript $\rho_n$ in the coefficient for simplicity when no confusion can arise.
\begin{definition}
When the twisted chain complex $C_*(W; V_n)$ is acyclic, 
we define the higher dimensional Reidemeister torsion 
for $W$ and $\rho_n$
as 
$\Tor{W}{V_n}$
and denote by $\Tor{W}{\rho_n}$ 
since the coefficient $V_n$ of $C_*(W; V_n)$ is determined by 
$\rho$ and $n$. 
\end{definition}

Increasing $n$ to infinity, we obtain the sequence of the Reidemeister torsion 
$\Tor{W}{V_n}$ when $C_*(W; V_n)$ is always acyclic.
We will observe the asymptotic behaviors of these sequences for Seifert fibered spaces in the subsequent sections. 

We also review eigenvalues of the image $\sigma_n (A)$ for $A \in \SL$.
Let $a^{\pm 1}$ be the eigenvalues of $A$.
By direct calculation, we can see that 
the eigenvalues of $\sigma_n(A) \in \SL[n]$ are given by 
$a^{-n+1}, a^{-n+3}, \ldots, a^{n-1}$,
\ie 
the weight space of $\sigma_n$ is $\{-n+1, -n+3, \ldots, n-1\}$ and
the multiplicity of each weight is $1$.

\begin{remark}
  \label{remark:weight}
  If $n > 1$ is even (resp. odd), the eigenvalues of $\sigma_n(A)$ are
  the odd (resp. even) powers, by odd (resp. even) integers from $1$ (resp. $0$) to $n-1$,
  of the eigenvalues of $A$.
  This implies that $\det \sigma_n(A) = 1$ for any $A \in \SL$ and $n \geq 1$.
\end{remark}

\section{Surgery formula for the asymptotics of higher dimensional Reidemeister torsion}
\label{sec:surgery_formula}
We will give a surgery formula 
to observe the asymptotic behavior of the sequence given by the higher dimensional Reidemeister torsions.
Our surgery formula is based on Lemma~\ref{lemma:MultLemma} (Multiplicativity Lemma)
in Section~\ref{section:Rtorsion}.
We will consider a connected compact orientable $3$-manifold $\mfd$ with torus boundary 
$\partial \mfd = T^2_1 \cup \ldots \cup T^2_m$.
We denote a pair of meridian and longitude on $T^2_j$ by $(q_j, h_j)$, 
\ie $\pi_1(T^2_j) = \langle q_j, h_j \,|\, [q_j, h_j] = 1\rangle$.
By Dehn filling with slopes $\alpha_1 / \beta_1, \ldots, \alpha_m / \beta_m$, 
we obtain a closed $3$-manifold $\resultingMfd$.
Here a slope $\alpha_j / \beta_j$ is the unoriented isotopy class of the essential simple loop
$\alpha_j q_j + \beta_j h_j$
on the $j$-th boundary component $T^2_j$, \ie
$$
\resultingMfd 
= \mfd \cup (\cup_{j=1}^m D^2_j \times S^1_j)
\quad \text{where}\quad \partial D^2_j \times \{*\} \sim \alpha_j q_j + \beta_j h_j \quad (\forall j).
$$

Our purpose is to express the higher dimensional Reidemeister torsion of resulting manifolds
$\resultingMfd$ by those of 
$\mfd$ and solid tori $S_j = D^2_j \times S^1_j$ $(j=1, \ldots, m)$. 
We start with a homomorphism $\rho$ from $\pi_1(\mfd)$ to $\SL$.
When $\rho$ satisfies that the equations $\rho(q_j)^{\alpha_j} \rho(h_j)^{\beta_j} = \I$, 
it extends to a homomorphism of the fundamental group of the resulting manifold and 
also defines the higher dimensional representations $\rho_n$ of $\pi_1(\resultingMfd)$.
Then we can consider the short exact sequence with the coefficient $V_n$:
\begin{equation}
  \label{eqn:MayerVietorisV_n}
  0 \to \oplus_{j=1}^m C_*(T^2_j;V_n)
  \to C_*(\mfd;V_n) \oplus (\oplus_{j=1}^m C_*(S_j;V_n)) 
  \to C_*(\resultingMfd;V_n )
  \to 0.
\end{equation}

If the left and middle parts in the short exact sequence~\eqref{eqn:MayerVietorisV_n} are acyclic, then 
the higher dimensional Reidemeister torsion of $\resultingMfd$
is expressed as
\begin{equation}
\label{eqn:HighRTorSurgeredMfd}
\Tor{\resultingMfd}{\rho_n}
=
\pm 
\Tor{\mfd}{\rho_n} \cdot
\prod_{j=1}^m \Tor{S_j}{\rho_n} \cdot 
\prod_{j=1}^m \Tor{T^2_j}{\rho_n}^{-1}.
\end{equation}
by Lemma~\ref{lemma:MultLemma} (Multiplicativity Lemma). 

We have seen that if $n$ is odd, then the image $\sigma_n(A)$ always has the eigenvalue $1$
for any $A \in \SL$.
This implies that the twisted chain complex $C_*(S_j;V_n)$ is never acyclic when $n=2N-1$
(this will be seen in the following Subsection~\ref{subsec:acyc_condition}).
Hence we will focus on even dimensional representations $\rho_{2N}$
to apply Multiplicativity Lemma for acyclic Reidemeister torsions.

\medskip

In Subsection~\ref{subsec:acyc_condition}, we will give equivalent conditions 
for the twisted chain complexes for $T^2_j$ and $S_j$ to be acyclic for all $2N$.
At least, we have to work on our surgery formula under the resulting conditions.
We will derive from Eq.~\eqref{eqn:HighRTorSurgeredMfd}
a surgery formula for the asymptotic behaviors of the sequences obtained by
the higher dimensional Reidemeister torsion of $\resultingMfd$
and $\mfd$ in Subsection~\ref{subsec:surgery_formula}. 
The last Subsection~\ref{subsec:example_torusknots} gives the example of our surgery formula
in the case that $\mfd$ is a torus knot exterior.

\subsection{Acyclicity conditions for the boundary and solid tori}
\label{subsec:acyc_condition}
First, we review the twisted chain complexes of $T^2$ and $D^2 \times S^1$.
The torsion of $D^2 \times S^1$ coincides with that of core $\{0\} \times S^1$
since they are simple homotopy equivalent.
We consider the twisted chain complexes of $S^1$ instead of $D^2 \times S^1$.
Under the cell decomposition
$$
S^1 = e^0 \cup e^1 \quad \text{and} \quad T^2 = e^0 \cup e^1_1 \cup e^1_2 \cup e^2,
$$
the twisted chain complexes with the coefficient $V_n$ of $T^2$ and $S^1$ are described as follows:
\begin{align*}
  C_*(S^1; V_n) &\co 
  0 \to C_1(S^1 ;V_n)=V_n \xrightarrow{L - \I} C_0(S^1;V_n)=V_n \to 0, \\
  C_*(T^2; V_n) &\co 
  0 \to C_2(T^2;V_n) = V_n \xrightarrow{\bnd{2}} C_1(T^2 ;V_n)=V_n^{\oplus 2} \xrightarrow{\bnd{1}} C_0(T^2 ;V_n)=V_n \to 0, \\
  & \qquad
  \bnd{2}=
  \begin{pmatrix}
    -H + \I \\
    Q - \I
  \end{pmatrix}, \quad 
  \bnd{1} = (Q -\I, H-\I)
\end{align*}
where $L$, $Q$ and $H$ denote $\SL[n]$-matrices corresponding
the simple closed loops $\ell = e^0 \cup e^1$ in $S^1$, 
$q = e^0 \cup e^1_1$ and $h = e^0 \cup e^1_2$ in $T^2$.

The twisted homology group $H_1(S^1;V_n)$ is the eigenspace of $L$ 
for the eigenvalue $1$.
As mentioned in Remark~\ref{remark:weight},
if $n$ is odd,
then the $\SL[n]$-matrix $L$ always has 
the eigenvalue $1$.
Hence $C_*(S^1;V_{2N-1})$ can not be acyclic.
Here and subsequently, we focus only on the twisted chain complexes given by 
the even dimensional vector spaces $V_{2N}$.

It is known that every abelian subgroup in $\SL$ is moved by conjugation into 
either the maximal abelian subgroups $\Hyp$ or $\Para$:
$$
  \Hyp := \left\{\left.
  \begin{pmatrix}
    z & 0 \\
    0 & z^{-1}
  \end{pmatrix}
  \,\right|\,
  z \in \C \setminus\{0\}
  \right\},\,
  \Para :=
  \left\{\left.
  \begin{pmatrix}
    \pm 1 & w \\
    0 & \pm 1
  \end{pmatrix}
  \,\right|\,
  w \in \C
  \right\}.
$$

Since the conjugation of representations induces an isomorphism between twisted homology groups,
we can assume that $\SL$-representations send $\pi_1(S^1)$ and $\pi_1(T^2)$ into
$\Hyp$ or $\Para$.

We describe the acyclicity conditions 
for the twisted chain complexes $C_*(S^1;V_{2N})$ and $C_*(T^2;V_{2N})$
by the terminologies of the $\SL$-matrices corresponding to 
generators of $\pi_1(S^1)$ and $\pi_1(T^2)$.

\begin{proposition}
  \label{prop:acyclicity_circle}
  Let $\rho$ be an $\SL$-representation of
  $\pi_1(S^1) = \langle \ell \rangle$.
  The composition $\rho_{2N} = \sigma_{2N} \circ \rho$ is acyclic for all $N \geq 1$
  if and only if
  $\rho(\ell)$ is neither of odd order nor parabolic with the trace $2$.
\end{proposition}
\begin{proof}
  The dimensions of $H_0(S^1;V_{2N})$ and $H_1(S^1;V_{2N})$ are same 
  since the Euler characteristic of $S^1$ is zero.
  The homology group $H_1(S^1;V_{2N})$ is the eigenspace of 
  $L=\rho_{2N}(\ell)$ for the eigenvalue $1$.
  The acyclicity of $\rho_{2N}$ is equivalent for the $\SL[2N]$-matrix $L$ not to have the eigenvalue $1$.
  
  If $\rho(\ell) \in \Hyp$ is not of odd order, 
  then the eigenvalues of $L$ forms 
  $\{e_{\ell}^{\pm (2k-1)} \,|\, k=1, \ldots, N\}$
  where $e_\ell^{\pm 1}$ are the eigenvalues of $\rho(\ell)$.
  Since $e_\ell$ is not of odd order, the $\SL[2N]$-element $L$ does not have the eigenvalue $1$ for all $N$.
  Thus $\rho_{2N}$ is acyclic for all $N$.
  If $\rho(\ell) \in \Para$ has the trace $-2$, then the eigenvalue of $L$ is just $-1$ for all $N$.
  The $\SL[2N]$-representation $\rho_{2N}$ is also acyclic for all $N$.

  Conversely suppose that $\rho(\ell)$ has the order $2 k_\ell -1$.
  Then the set of eigenvalues of $L$ contain $1$ when $N \geq k_\ell$.
  The twisted homology group $H_1(S^1;V_{2N})$ is not trivial for $N \geq k_\ell$.
  Suppose that $\rho(\ell) \in \Para$ has the trace $2$.
  Then the eigenvalue of $L$ is just $1$ for all $N$.
  The twisted homology group $H_1(S^1;V_{2N})$ is not trivial for all $N$.
\end{proof}

\begin{proposition}
  \label{prop:acyclicity_torus}
  Let $\rho$ be an $\SL$-representation of 
  $\pi_1(T^2) = \langle q, h \,|\, [q, h]=1\rangle$.
  The composition $\rho_{2N} = \sigma_{2N} \circ \rho$ is acyclic
  if and only if
  either $\rho(q)$ or $\rho(h)$ is neither 
  of odd order nor parabolic with the trace $2$.
\end{proposition}
\begin{proof}
  The homology group $H_2(T^2;V_{2N})$ is generated by the common eigenvectors of 
  $Q=\rho_{2N}(q)$ and $H=\rho_{2N}(h)$ for the eigenvalue $1$. 
  Since the Euler characteristic of $T^2$ is zero, by Poincar\'e duality,
  the twisted homology group $H_*(T^2;V_{2N})$ vanishes if and only if
  $H_2(T^2;V_{2N}) = \bm{0}$.
  The acyclicity of $\rho_{2N}$ is equivalent to exist no common eigenvectors
  of $Q$ and $H$ for the eigenvalue $1$.

  If $\rho(q)$ is neither of odd order nor parabolic with the trace $2$, then
  $Q$ does not have the eigenvalue $1$ for all $N$
  as in the proof of Proposition~\ref{prop:acyclicity_circle}.
  We have no common eigenvectors of $Q$ and $H$ for the eigenvalue $1$.
  Hence $\rho_{2N}$ is acyclic for all $N$.

  Conversely suppose that $\rho(q)$ and $\rho(h)$ have the orders $2 k_q -1$ and $2k_h-1$.
  Then at the weight $(2k_q -1)(2k_h-1)$, we have a common eigenvector of $Q$ and $H$ for the eigenvalue $1$
  when $N$ is sufficiently large.
  Thus $\rho_{2N}$ is not acyclic for sufficiently large $N$.
  Suppose that $\rho(q)$ and $\rho(h)$ are parabolic with the trace $2$.
  Then $Q$ and $H$ are always upper triangular matrix whose all diagonal entries are $1$.
  We have a common eigenvector of $Q$ and $H$ for the eigenvalue $1$.
  Hence $\rho_{2N}$ is not acyclic for all $N$.
\end{proof}

\begin{remark}
  One can show that the $\SL[2N-1]$-representation $\rho_{2N-1}$ of $\pi_1(T^2)$
  is not acyclic for all $N \geq 1$
  by the similar argument in Proposition~\ref{prop:acyclicity_torus}.
\end{remark}

\subsection{Surgery formula for the asymptotic behaviors}
\label{subsec:surgery_formula}
We show a surgery formula for the asymptotic behaviors
of the higher dimensional Reidemeister torsions
of a closed $3$-manifold 
$\resultingMfd
= \mfd \cup (\cup_{j=1}^m) D^2_j \times S^1_j$.
To apply Lemma~\ref{lemma:MultLemma} (Multiplicativity Lemma) for acyclic chain complexes,
we assume that the following acyclicity conditions for 
the twisted chain complexes $\partial \mfd = \cup_{j=1}^m T^2_j$ and 
solid tori $S_j = D^2_j \times S^2_j$ with the presentations of fundamental groups:
$$
  \pi_1(T^2_j) = \langle q_j, h_j \,|\, [q_j, h_j] = 1 \rangle, \quad 
  \pi_1(S_j) = \langle \ell_j \rangle.
$$

\begin{definition}[Acyclicity conditions]
  \label{def:acyclicity_conditions}
  Let $\rho$ be a homomorphism from $\pi_1(\mfd)$ to $\SL$
  such that $\rho(q_j^{\alpha_j} h_j^{\beta_j}) = \I$ for all $j=1, \ldots m$.
  We use the same symbol $\rho$ for the induced homomorphism of
  $\pi_1(\resultingMfd)$ and assume that for all $j = 1, \ldots, m$
  \begin{enumerate}
  \item
    \label{item:acyclicity_boundary}
    either $\rho(q_j)$ or $\rho(h_j)$ is of even order and;
  \item
    \label{item:acyclicity_solid_torus}
    the order of $\rho(\ell_j)$ is also even.
  \end{enumerate}
  We will call the above conditions~\eqref{item:acyclicity_boundary} \&~\eqref{item:acyclicity_solid_torus}
  {\it the acyclicity conditions}.
\end{definition}

\begin{remark}
  The acyclicity conditions guarantee that all twisted chain complexes
  of $T^2_j$ and $S_j$ are acyclic.
  Our acyclicity conditions are more restricted as compared with
  the conditions in Propositions~\ref{prop:acyclicity_torus} \& \ref{prop:acyclicity_circle}.
  However in the case that the resulting manifold $\resultingMfd$ is a Seifert fibered space,
  it is reasonable to assume our conditions as seen in Section~\ref{sec:asymptotics_Seifert}.
\end{remark}

Under the acyclic conditions, if an $\SL[2N]$-representation $\rho_{2N}$ of $\pi_1(\mfd)$ is acyclic,
then we can express  the higher dimensional Reidemeister torsion of $\resultingMfd$ as
\begin{equation}
  \label{eqn:torsion_explicit_product}
  \Tor{\resultingMfd}{\rho_{2N}}
  = \Tor{\mfd}{\rho_{2N}} \cdot \prod_{j=1}^m \Tor{S_j}{\rho_{2N}} \cdot \prod_{j=1}^m \Tor{T^2_j}{\rho_{2N}}^{-1}.
\end{equation}
by applying Lemma~\ref{lemma:MultLemma} (Multiplicativity Lemma).
Note that every integer $\beta'_i$ in Lemma~\ref{lemma:MultLemma}
is even from the acyclicity of $C_*(T^2_j;V_{2N})$.

Then the asymptotics of $\log|\Tor{\resultingMfd}{\rho_{2N}}|$ for $2N$ is determined by 
that of $\log|\Tor{\mfd}{\rho_{2N}}|$ as follows.

\begin{theorem}
  \label{thm:surgery_formula}
  Let $\rho$ be an $\SL$-representation of $\pi_1(\mfd)$ satisfying
  $\rho(q_j^{\alpha_j} h_j^{\beta_j}) = \I$ and
  the acyclicity conditions in Definition~\ref{def:acyclicity_conditions}.
  Suppose that $\rho_{2N}$ of $\pi_1(\mfd)$ is acyclic for all $N$.
  Then the asymptotics of $\log|\Tor{\resultingMfd}{\rho_{2N}}|$ is expressed as follows:
  \begin{enumerate}
  \item
    $\displaystyle{
    \lim_{N \to \infty}
    \frac{
      \log|\Tor{\resultingMfd}{\rho_{2N}}|
    }{
      (2N)^2
    }
    =
    \lim_{N \to \infty}
    \frac{
      \log|\Tor{\mfd}{\rho_{2N}}|
    }{
      (2N)^2
    },
    }$
  \item
    $\displaystyle{
    \lim_{N \to \infty}
    \frac{
      \log|\Tor{\resultingMfd}{\rho_{2N}}|
    }{
      2N
    }
    =
    \lim_{N \to \infty}
    \frac{
      \log|\Tor{\mfd}{\rho_{2N}}|
    }{
      2N
    }
    -\log 2 \sum_{j=1}^m \frac{1}{\lambda_j}
    }$
  \end{enumerate}
  where $2\lambda_j$ is the order of $\rho(\ell_j)$ and
  $\ell_j$ is the homotopy class of $\{0\} \times S^1_j \subset S_j$.
\end{theorem}
\begin{proof}
  By Eq.~\eqref{eqn:torsion_explicit_product},
  the logarithm $\log |\Tor{\resultingMfd}{\rho_{2N}}|$ is expressed as
  \begin{align*}
  &\log |\Tor{\resultingMfd}{\rho_{2N}}| \\
  & \quad= \log|\Tor{\mfd}{\rho_{2N}}|
  + \sum_{J=1}^m \log|\Tor{S_j}{\rho_{2N}}|
  - \sum_{j=1}^m \log|\Tor{T^2_j}{\rho_{2N}}|.
  \end{align*}
  Applying the following Propositions~\ref{prop:torsion_torus} \&~\ref{prop:torsion_circle},
  we obtain Theorem~\ref{thm:surgery_formula}.   
\end{proof}
\begin{proposition}
  \label{prop:torsion_torus}
  Let $\rho$ be an $\SL$-representation of
  $\pi_1(T^2)=\langle q, h \,|\, [q, h] =1\rangle$
  such that either $\rho(q)$ or $\rho(h)$ is of even order.
  Then $\Tor{T^2}{\rho_{2N}} = 1$ for all $N \geq 1$.
\end{proposition}
\begin{proof}
  This follows from the direct calculation.
\end{proof}

\begin{proposition}
  \label{prop:torsion_circle}
  Let $\rho$ be an $\SL$-representation of
  $\pi_1(S^1)=\langle \ell \rangle$
  such that $\rho(\ell)$ has an even order.
  Then we have the following limits of $\log|\Tor{S^1}{\rho_{2N}}|$:
  \begin{enumerate}
  \item
    \label{item:torsion_circle_square_order}
    $\displaystyle{\lim_{N \to \infty}
    \frac{\log|\Tor{S^1}{\rho_{2N}}|}{(2N)^2} = 0}$,
  \item
    \label{item:leading_coeff_circle}
    $\displaystyle{\lim_{N \to \infty}
    \frac{\log|\Tor{S^1}{\rho_{2N}}|}{2N} = \frac{-\log 2}{\lambda}}$
  \end{enumerate}
  where $2\lambda$ is the order of $\rho(\ell)$.
\end{proposition}
\begin{proof}
  We begin with computing the Reidemeister torsion
  $\Tor{S^1}{\rho_{2N}}$.
  We can express the Reidemeister torsion as
  $
  \Tor{S^1}{\rho_{2N}} = \det (\rho_{2N}(\ell) - \I)^{-1}.
  $
  The set of the eigenvalues of $\rho_{2N}(\ell)$
  is obtained from  
  the eigenvalues $e^{\pm \pi \eta \sqrt{-1}/\lambda}$ 
  of $\rho(\ell)$, 
  where $\eta$ is odd and $\eta$ and $\lambda$ are coprime.
  It turns into  $\{e^{\pm \pi (2k-1) \eta \sqrt{-1}/\lambda} \,|\, k = 1, \ldots, N\}$.
  Hence the Reidemeister torsion $\Tor{S^1}{\rho_{2N}}$ turns out
  \begin{align*}
    \Tor{S^1}{\rho_{2N}}
    &=
    \prod_{k=1}^m \{(e^{\pi (2k-1) \eta \sqrt{-1} / \lambda} -1)((e^{-\pi (2k-1) \eta \sqrt{-1} / \lambda} -1))\}^{-1} \\
    &=
    \prod_{k=1}^N \left(2 \sin \frac{\pi (2k-1)\eta}{2\lambda} \right)^{-2}.
  \end{align*}
  The logarithm $\log |\Tor{S^1}{\rho_{2N}}|$ is expressed as
  \begin{equation}
    \label{eqn:log_torsion_circle}
    \log |\Tor{S^1}{\rho_{2N}}| 
    = 2N \log 2^{-1} + 2 \sum_{k=1}^N \log \left|\sin \frac{\pi (2k-1)\eta}{2\lambda} \right|^{-1}.
  \end{equation}
  We can now proceed to compute
  the limits~\eqref{item:torsion_circle_square_order} \&~\eqref{item:leading_coeff_circle}.

  \eqref{item:torsion_circle_square_order}\,
  From the following inequality
  $$
  \left|\sin \frac{\pi}{2 \lambda} \right|
  \leq \left| \sin \frac{\pi (2k-1)\eta}{2 \lambda} \right|
  \leq 1
  $$
  it follows that
  \begin{equation}
    \label{eqn:ineq_log_circle}
    N \log \left|\sin \frac{\pi}{2 \lambda} \right|^{-1}
    \geq
    \sum_{k=1}^N \log \left| \sin \frac{\pi (2k-1)\eta}{2 \lambda} \right|^{-1}
    \geq 0.
  \end{equation}
  By the inequality~\eqref{eqn:ineq_log_circle} and explicit form~\eqref{eqn:log_torsion_circle}
  of $\log |\Tor{S^1}{\rho_{2N}}|$,
  we can assert
  $$\lim_{N \to \infty}
  \frac{\log|\Tor{S^1}{\rho_{2N}}|}{(2N)^2} = 0.
  $$

  \eqref{item:leading_coeff_circle}\, 
  We can express the limit~\eqref{item:leading_coeff_circle} as
  \begin{equation}
    \label{eqn:leading_term_circle}
    \lim_{N \to \infty}
    \frac{\log|\Tor{S^1}{\rho_{2N}}|}{2N}
    = \log 2^{-1} +
    \lim_{N \to \infty}
    \frac{1}{N} \sum_{k=1}^{N} \log \left| \sin \frac{\pi (2k-1)\eta}{2\lambda} \right|^{-1}.
  \end{equation}
  The second term in the right hand side of~\eqref{eqn:leading_term_circle} can be rewritten as
  $$
  \lim_{N \to \infty}
  \frac{1}{N} \sum_{k=1}^{N} \log \left| \sin \frac{\pi (2k-1)\eta}{2\lambda} \right|^{-1}
  = 
  \lim_{N \to \infty}
  \frac{1}{N} \sum_{k=1}^{N}
  \log \left| \sin \left( \frac{\pi \eta}{2\lambda} + \frac{\pi (k-1)\eta}{\lambda} \right)\right|^{-1}
  $$
  The sequence 
  $\{\log \left| \sin \left(\pi \eta / (2\lambda) + \pi (k-1)\eta /\lambda \right)\right|^{-1}\}_{k=1, 2, \ldots}$
  has the minimum period $\lambda$ since $\eta$ and $\lambda$ are coprime.
  By Lemma~\ref{lemma:limit_average}, we can rewrite as 
  $$
  \lim_{N \to \infty}
  \frac{1}{N} \sum_{k=1}^{N}
  \log \left| \sin \left( \frac{\pi \eta}{2\lambda} + \frac{\pi (k-1)\eta}{\lambda} \right)\right|^{-1}
  = 
  \frac{1}{\lambda} \sum_{k=1}^\lambda 
  \log \left| \sin \left( \frac{\pi \eta}{2\lambda} + \frac{\pi (k-1)\eta}{\lambda} \right)\right|^{-1}.
  $$
  The right hand side of~\eqref{eqn:leading_term_circle} turns into  
  \begin{align*}
    \log 2^{-1} +
    \frac{1}{\lambda} \sum_{k=1}^{\lambda}
    \log \left| \sin \left( \frac{\pi \eta}{2\lambda} + \frac{\pi (k-1)\eta}{\lambda} \right)\right|^{-1}
    &=
    \frac{1}{\lambda} 
    \log \prod_{k=1}^{\lambda}
    \left| 2\sin \left( \frac{\pi \eta}{2\lambda} + \frac{\pi (k-1)\eta}{\lambda} \right)\right|^{-1} \\
    &=
    \frac{1}{\lambda} 
    \log \left| 2\sin \left( \frac{\pi \eta}{2} \right)\right|^{-1}
  \end{align*}
  by $|2\sin (n\theta)| = \prod_{k=0}^{n-1} |2 \sin (\theta + k\pi/n)|$.
  Therefore we obtain the limit
  $$
  \lim_{N \to \infty}
  \frac{\log|\Tor{S^1}{\rho_{2N}}|}{2N}
  = \frac{-\log 2}{\lambda}
  $$
  since $\eta$ is odd.
\end{proof}

\begin{remark}
  We obtain $|2\sin (n\theta)| = \prod_{k=0}^{n-1} |2 \sin (\theta + k\pi/n)|$
  from substituting $z = e^{-2\theta \sqrt{-1}}$ to $|z^n - 1| = \prod_{k=0}^{n-1} | z - e^{2\pi k\sqrt{-1} / n}|$.
\end{remark}

\begin{lemma}
  \label{lemma:limit_average}
  Let $\{a_k\,|\, a_k \in \R\}_{k \geq 1}$ be a sequence such that  
  $a_k \geq 0$ and $a_{k+N_0} = a_k$.
  Then we have the following limit:
  $$
  \lim_{N \to \infty} 
  \frac{a_1 + \cdots + a_N}{N}
  = \frac{
    a_1 + \cdots + a_{N_0}}{N_0}.
  $$
\end{lemma}
\begin{proof}
  It follows that 
  $
  \left[\frac{N}{N_0}\right] \sum_{k=1}^{N_0} a_k
  \leq
  \sum_{k=1}^{N} a_k
  \leq 
  \left(\left[\frac{N}{N_0}\right] + 1\right)
  \sum_{k=1}^{N_0} a_k
  $
  where $[\,x\,]$ denotes the maximal integer less than or equal to $x$.
  Note that $\frac{N}{N_0} -  1 < \left[\frac{N}{N_0}\right] \leq \frac{N}{N_0}$.
\end{proof}

\begin{remark}
  In Lemma~\ref{lemma:limit_average}, 
  it is not required that the period $N_0$ is minimum.
  However we have the same average for any period $N_0$.
\end{remark}

\subsection{Example for Dehn fillings of torus knot exteriors}
\label{subsec:example_torusknots}
We give examples of Theorem~\ref{thm:surgery_formula} 
for integral surgeries along torus knots in $S^3$.
Let $\mfd$ be the $(p, q)$-torus knot exterior
which is obtained by removing an open tubular neighbourhood of the knot from $S^3$.
After gluing a solid torus along the slope $1/n$ $(n \in \Z)$ on $\partial \mfd$, 
we have an integral homology $3$-sphere $\mfd(\frac{1}{n})$.
Since we consider the $(p, q)$-torus knot exterior, 
the resulting manifold $\mfd(\frac{1}{n})$ is a Brieskorn homology $3$-sphere
of index $(p, q, pqn \pm 1)$.
Here the sign in $pqn \pm 1$ depends on the orientation of the preferred longitude on $\partial \mfd$.

The $(p, q)$-torus knot group admits the following presentation:
$$
\pi_1(\mfd) =
\langle x, y \,|\, x^p=y^q \rangle.
$$
In this presentation, we can express a pair of meridian $m$ and longitude $\ell$ as
$$
m = x^{-u} y^v, \quad \ell = m^{pq}x^{-p} 
$$
where $u$ and $v$ are integers satisfying that $pv-qu=1$.
Then the Brieskorn homology sphere $\mfd(\frac{1}{n})$ has the index $(p, q, pqn+1)$.
We consider irreducible $\SL$-representations $\rho$ of $\pi_1(\mfd)$ such that 
$\rho(m \ell^{n}) = \I$,
\ie
they extend to irreducible $\SL$-representations of
$\pi_1(\mfd(\frac{1}{n})) = \langle x, y \,|\, x^p=y^q, m\ell^n=1 \rangle$.
Here irreducible means that
there are no common non--trivial eigenvectors among all elements in $\rho(\pi_1(\mfd))$.
Under the assumption of irreducibility for $\rho$,
the central element $x^p (=y^q)$ must be sent to $\pm \I$.
The requirement that $\rho(m \ell^n) = \I$ turns into $\rho(m)^{npq+1} = \pm \I$.
Hence we have the constrains on the order of $\rho(x)$, $\rho(y)$ and $\rho(m)$ 
for every irreducible $\SL$-representation of $\pi_1(\mfd(\frac{1}{n}))$.

The conjugacy classes of irreducible $\SL$-representations of $\pi_1(\mfd(\frac{1}{n}))$
form a finite set. Each member of the finite set 
corresponds to a triple of integers. This was shown by D.~Johnson~\cite{Johnson:unpublished}
and he also gave
the explicit form of Reidemeister torsion for acyclic $\SL$-representations
as follows.
\begin{theorem}[Johnson]
  The conjugacy classes of irreducible $\SL$-representations $\rho$ of $\pi_1(\mfd(\frac{1}{n}))$
  are given by triples $(a, b, c)$ such that 
  \begin{enumerate}
  \item
    $0 < a < p$, $0 < b < q$, $a \equiv b \quad \mathrm{mod}\, 2$,
  \item
    $0 < c < r =|pqn+1|$, $c \equiv na \quad \mathrm{mod}\, 2$,
  \item
    $\trace \rho(x) = 2 \cos \pi a / p$,
  \item
    $\trace \rho(y) = 2 \cos \pi b / q$,
  \item
    $\trace \rho(m) = 2 \cos \pi c / r$.
  \end{enumerate}
  The Reidemeister torsion is given by
  $$
  \Tor{\mfd(\hbox{$\frac{1}{n}$})}{\rho} =
  \begin{cases}
    2^{-4} \sin^{-2} \frac{\pi a}{2p} \sin^{-2} \frac{\pi b}{2q} \sin^{-2} \frac{\pi(cpq - r)}{2r}
    & a \equiv b \equiv 1, c \equiv n \quad \mathrm{mod}\,2\\
    \text{non-acyclic} 
    & a \equiv b \equiv 0\,  or\,  c \not \equiv n \quad \mathrm{mod}\,2
    \end{cases}
  $$
  for $\rho \in (a, b, c)$.
\end{theorem}
In the remainder of this subsection, we denote by $(a, b, c)$ the corresponding conjugacy class of irreducible 
$\SL$-representations.
We also refer to~\cite{Freed:Brieskorn} for the Reidemeister torsion of Brieskorn homology $3$-spheres.
\begin{remark}
  The parameters $a$ and $b$ determine the image of the central element $x^p (=y^q)$ by $(-\I)^a (=(-I)^b)$.
\end{remark}

To apply Theorem~\ref{thm:surgery_formula}, we need to find
\begin{itemize}
\item a condition on $(a, b, c)$ for all $\rho_{2N}|_{\pi_1(\mfd)}$ to be acyclic and
\item the orders of $\SL$-elements in the acyclicity conditions (Definition~\ref{def:acyclicity_conditions}).
\end{itemize}
The author has shown in~\cite[Proposition~3.1]{yamaguchi:RtorTorusKnots} that 
$\rho|_{\pi_1(\mfd)}$ induces an 
acyclic $\SL[2N]$-representation for all $N$ 
if and only if the parameters $a$ and $b$ of $\rho|_{\pi_1(\mfd)}$ satisfy that  
$a \equiv b \equiv 1$ mod $2$.

Since the surgery slope for $\mfd(\frac{1}{n}) = \mfd \cup D^2 \times S^1$ is $1/n$, 
the homotopy class of the core $\{0\} \times S^1$ is given by $\ell^{\pm 1}$ in $\pi_1(\mfd(\frac{1}{n}))$.
To check the acyclicity conditions, we only need to find the order of $\rho(\ell)$.
Let $\rho$ be in the conjugacy class $(a, b, c)$ such that $a \equiv b \equiv 1$ mod $2$.
Then it follows from $c \equiv na$ mod 2 that  
$$\rho(\ell)^r = (-\I)^{pqc - r} = -\I.$$
Hence the eigenvalues of $\rho(\ell)$ are given by $e^{\pm \pi \eta \sqrt{-1}/r}$ for some odd integer $\eta$,
which shows that $\rho(\ell)$ has an even order.

Let us apply Theorem~\ref{thm:surgery_formula} to the Brieskorn homology $3$-sphere
$\mfd(\frac{1}{n})$.
We obtain 
\begin{align}
  \lim_{N \to \infty}
  \frac{
    \log| \Tor{\mfd(\frac{1}{n})}{\rho_{2N}} |
  }{
    (2N)^2
  }
  &=
  \lim_{N \to \infty}
  \frac{
    \log| \Tor{\mfd}{\rho_{2N}} |
  }{
    (2N)^2
  } \label{eqn:torus_knot_squared}\\
  \lim_{N \to \infty}
  \frac{
    \log | \Tor{\mfd(\frac{1}{n})}{\rho_{2N}} |
  }{
    2N
  }
  &=
  \lim_{N \to \infty}
  \frac{
    \log | \Tor{\mfd}{\rho_{2N}} |
  }{
    2N
  }
  - \frac{\log 2}{r'}
  \label{eqn:torus_knot_leading}
\end{align}
where $2r'$ is the order of $\rho(\ell)$.
Note that it is seen from the g.c.d. $(p, r) = (q, r) = 1$ that $r' = r / (c, r)$.

It has shown in~\cite[Theorem~4.2]{yamaguchi:RtorTorusKnots} that 
the right hand side in \eqref{eqn:torus_knot_squared} vanishes.
The higher Reidemeister torsion of the torus knot exterior $\mfd$ is expressed as,
by~\cite[Proposition~4.1]{yamaguchi:RtorTorusKnots},
$$
\Tor{\mfd}{\rho_{2N}}
=
\frac{2^{2N}}{
\prod_{k=1}^N 4^2 \sin^2 \frac{\pi (2k-1)a}{2p} \sin^2 \frac{\pi (2k-1)b}{2q}}.
$$
By a similar argument to the proof of Proposition~\ref{prop:torsion_circle},
we can see that the limit in the right hand of~\eqref{eqn:torus_knot_leading}
turns out 
$$
\lim_{N \to \infty}
\frac{
  \log | \Tor{\mfd}{\rho_{2N}} |
}{
  2N
}
= \left( 1 - \frac{1}{p'} - \frac{1}{q'} \right)\log 2
$$
where $p' = p / (p,a)$ and $q' = q / (q,b)$. 

\begin{theorem}
  \label{thm:asymptotics_Brieskorn}
  The growth of $\log |\Tor{\mfd(\frac{1}{n})}{\rho_{2N}}|$ has the same order as $2N$
  for any acyclic irreducible representation $\rho$ of $\pi_1(\mfd(\frac{1}{n}))$.
  Moreover if $\rho$ is contained in the conjugacy class $(a, b, c)$,
  then the leading coefficient in $2N$
  converges as 
  $$
  \lim_{N \to \infty}
  \frac{
    \log | \Tor{\mfd(\frac{1}{n})}{\rho_{2N}} |
  }{
    2N
  }
  =
  \left(
  1 - \frac{1}{p'} - \frac{1}{q'} - \frac{1}{r'}
  \right)\log 2
  $$
  where $p' = p / (a, p)$, $q' = q / (b, q)$ and $r' = r / (c, r)$.

  In the case that $(a, p) = (b, q) = (c, r) = 1$,
  the leading coefficient converges to 
  the maximum $(1 - 1/p - 1/q - 1/r) \log2$
\end{theorem}
For more details on the limits of $\log | \Tor{\mfd}{\rho_{2N}} | /  (2N)$, 
we refer to~\cite{Yamaguchi13:kokyuroku}.

\section{Asymptotics of higher dimensional Reidemeister torsion for Seifert fibered spaces}
\label{sec:asymptotics_Seifert}
We will apply Theorem~\ref{thm:surgery_formula} to Seifert fibered spaces and 
study the asymptotic behaviors of their higher dimensional Reidemeister torsions.
We will see the growth of the logarithm of the higher dimensional Reidemeister torsion
has the same order as the dimension of representation and 
we will also give the explicit limit of 
the leading coefficient.

The limits of the leading coefficients are determined by each component in 
the $\SL$-representation space of the fundamental group of a Seifert fibered space,
that is to say,
we obtain a locally constant function on the $\SL$-representation space.
From the invariance of Reidemeister torsion under the conjugation of representations,
we also obtain a locally constant function on the character variety of 
the fundamental group of a Seifert fibered space.

We will focus on $\SU$-character varieties for Seifert fibered homology spheres and 
describe explicit values of the locally constant functions.
Our calculation shows that these locally constant functions 
take the maximum values on the top dimensional components
and the explicit maximums are given by $-\chi \log 2$
where $\chi$ is the Euler characteristic of the base orbifold 
of a Seifert fibered homology sphere.

We start with a brief review on Seifert fibered spaces in Subsection~\ref{subsec:Seifert}.
Subsection~\ref{subsec:asymptotics_Seifert} shows
the application of Theorem~\ref{thm:surgery_formula}
to Seifert fibered spaces.
We will observe the relation between limits of the leading coefficients and 
components in $\SU$-character varieties in Subsection~\ref{subsec:leadingCoeff_Seifert}.

\subsection{Seifert fibered spaces}
\label{subsec:Seifert}
A Seifert fibered space is referred as an $S^1$-fibration over a closed $2$-orbifold. 
We consider the orientable Seifert fibered space given by the following Seifert index:
$$
\{b, (o, g); (\alpha_1, \beta_1), \ldots, (\alpha_m, \beta_m)\}.
$$
where $\alpha_j \geq 2$ $(j=1,\ldots, m)$ and each pair of $\alpha_j$ and $\beta_j$ is coprime.
For a Seifert index, we refer to \cite{NeumannJankins:Seifert, Orlik:SeifertManifold}.

We can regard a Seifert fibered space  as an $S^1$-bundle
over a closed orientable surface $\Sigma$ with 
$m+1$ exceptional fibers, where the genus of $\Sigma$ is $g$.
From this viewpoint, we can decompose a Seifert fibered space into tubular neighbourhoods 
of exceptional fibers and their complement.
Set $\Sigma_* = \Sigma \setminus \mathrm{int}(D^2_0 \cup \ldots \cup D^2_m)$ where 
$D^2_0, \ldots, D^2_m$ are disjoint disks in $\Sigma$.
Let $\mfd$ be the trivial $S^1$-bundle $\Sigma_* \times S^1$.
We have a canonical decomposition of the Seifert fibered space as
the following union of $\mfd$ and solid tori:
\begin{align*}
  &\mfd \cup (S_0 \cup S_1 \cup \cdots \cup S_m)\\
  &= \Seifert.
\end{align*} 
The solid torus $S_0$ corresponds to the triviality obstruction $b$ and the others $S_j$ $(1 \leq j \leq m)$ 
correspond to the exceptional fibers with the index $(\alpha_j, \beta_j)$.
Then the fundamental group of $\Seifert$ admits the following presentation:
\begin{multline}
  \pi_1(\Seifert) \\
  = \langle
  a_1, b_1, \ldots, a_g, b_g, q_1, \ldots. q_m, h \,|\,
  [a_i, h] = [b_i, h] = [q_j, h] = 1, \\
  q_j^{\alpha_j} h^{\beta_j} = 1, 
  q_1 \cdots q_m [a_1, b_1] \cdots [a_g, b_g] = h^b
  \rangle
\end{multline}
where $a_i$ and $b_j$ correspond to generators of $\pi_1(\Sigma)$ and 
$q_j$ is the corresponding to the circle $\partial D^2_j \subset \Sigma$ and $h$ is 
the homotopy class of a regular fiber in $\mfd$.
Note that 
the presentation of $\pi_1(\mfd)$ is given by 
$\langle a_1, b_1, \ldots, a_g, b_g, q_0, q_1, \ldots. q_m, h \,|\,
[a_i, h] = [b_i, h] = [q_j, h] = 1,
q_1 \cdots q_m [a_1, b_1] \cdots [a_g, b_g] = q_0
\rangle$.

We review the acyclicity of $\SL[n]$-representations 
of the fundamental group of $\mfd$.
We are supposed to consider the sequences of Reidemeister torsions
for acyclic chain complexes $C_*(\mfd;V_{2N})$\, ($N \geq 1$)
derived from an $\SL$-representation $\rho$ of $\pi_1(\mfd)$.
It was shown in T.~Kitano~\cite{Kitano:RtorsionSeifertSL2}
that $C_*(\mfd;\C^n)$ is acyclic   
if and only if $\rho_n(h) = -\I$ where $(\C^n, \rho_n)$ is an $\SL[n]$-representation
of $\pi_1(\mfd)$.

We also touch the first homology groups of Seifert fibered spaces 
since we will consider Seifert fibered homology spheres in Subsection~\ref{subsec:leadingCoeff_Seifert}.
It is known that the first homology group of $\Seifert$ is 
expressed as $H_1(\Seifert;\Z) \simeq \Z^{2g}\oplus T$ where 
$T$ is a finite abelian group with the order $\alpha_1 \cdots \alpha_m |b + \sum_{j=1}^m \beta_j / \alpha_j|$
if $b + \sum_{j=1}^m \beta_j / \alpha_j$ is not zero.
In the case that $b + \sum_{j=1}^m \beta_j / \alpha_j = 0$, 
the homology group $H_1(\Seifert;\Z)$ has the free rank $2g+1$.
Hence, for any Seifert fibered homology sphere $\Seifert$,
the genus of the base orbifold is zero and we have the equation:
$$
\alpha_1 \cdots \alpha_m \Big(b + \sum_{j=1}^{m} \frac{\beta_j}{\alpha_j}\Big) = 1,
$$
in particular, which implies that $\alpha_j$ are pairwise coprime.

\subsection{The asymptotic behavior of the higher dimensional Reidemeister torsions for Seifert fibered spaces}
\label{subsec:asymptotics_Seifert}
We apply Theorem~\ref{thm:surgery_formula}
to a Seifert fibered space $\Seifert$
for $\SL$-representations $\rho$
which satisfy $\rho(h) = -\I$, 
that is to say, 
the central element $h$ is sent to the non-trivial central element $-\I$ of $\SL$.
We begin with observing the acyclicity for
$\mfd$, $\partial \mfd$ and the solid tori $S_j$
and confirm that the condition that $\rho(h) = -\I$ is only needed in
our situation.

It is shown from~\cite[the proof of Proposition~3.1]{Kitano:RtorsionSeifertSLn} 
that $C_*(\mfd;V_{2N})$ is acyclic if and only if $\rho_{2N}(h)=-\I$.
It holds for all $N$ that $\rho_{2N}(h) = -\I$
under the assumption that $\rho(h) = -\I$ since 
every weight of $\sigma_{2N}$ for the eigenvalues of $\rho(h)$ is odd.
Also we have shown that such an $\SL$-representation $\rho$ 
satisfies the acyclicity condition~\eqref{item:acyclicity_boundary} for all
boundary components $T^2_j$
since $\pi_1(T^2_j)$ is presented as 
$$
  \pi_1(T^2_j) =
  \langle q_j, h \,|\, [q_j, h] = 1 \rangle
$$
and $\rho(h)$ has the order of $2$.
Furthermore we can see that 
all conditions in our surgery formula (Theorem~\ref{thm:surgery_formula}) are satisfied
under the assumption that $\rho(h) = -\I$.
\begin{proposition}
  Suppose that an $\SL$-representation $\rho$ of $\pi_1(\Seifert)$ sends $h$ to $-\I$.
  Then $\rho$ satisfies the acyclicity conditions in Definition~\ref{def:acyclicity_conditions}
  and the restriction of $\rho$ gives an acyclic twisted chain complex $C_*(\mfd;V_{2N})$
  for any $N$.
\end{proposition}
\begin{proof}
  It remains to prove that the acyclicity for the twisted chain complexes of $S_j$
  (the condition~\eqref{item:acyclicity_solid_torus} in Definition~\ref{def:acyclicity_conditions}).
  This follows from the following Lemma~\ref{lemma:order_ell}.
\end{proof}

\begin{lemma}
  \label{lemma:order_ell}
  Let $\rho$ be an $\SL$-representations of $\pi_1(\Seifert)$ such that $\rho(h)=-\I$
  and $\ell_j$ denote the homotopy class of the core of $S_j$ for $j=0, \ldots, m$.
  Then $\rho(\ell_j)$ is of even order for all $j$.
\end{lemma}
\begin{proof}
  Set $\alpha_0 = 1$ and $\beta_0=-b$, which are the corresponding slope to the solid torus $S_0$.
  We can express each $\ell_j$ $(j=0, 1, \ldots, m)$ as 
  $
  \ell_j = q_j^{\mu_j} h^{\nu_j}
  $
  where integers $\mu_j$ and $\nu_j$ satisfy that 
  $\alpha_j \nu_j - \beta_j \mu_j = -1$ and $0 < \mu_j < \alpha_j$.
  We will show that every $\rho(\ell_j)^{\alpha_j}$ turns into $-\I$.
  For $j=0$, 
  the matrix $\rho(\ell_0)^{\alpha_0} (= \rho(\ell_0))$ turns out
  \begin{align*}
    \rho(\ell_0)^{\alpha_0}
    &= \rho(q_1 \cdots q_m [a_1, b_1]\cdots[a_g, b_g])^{\alpha_0 \mu_0} \rho(h)^{\alpha_0 \nu_0} \\
    &= \rho(h)^{b\mu_0+\nu_0} \\
    &= -\I.
  \end{align*}
  Similarly $\rho(\ell_j)^{\alpha_j}$ turns into
  $$
  \rho(q_j)^{\alpha_j \mu_j} \rho(h)^{\alpha_j \nu_j} = \rho(h)^{\alpha_j \nu_j - \beta_j \mu_j} = -\I.
  $$
  Hence, for all $j$,
  the eigenvalues of $\rho(\ell_j)$ is given by $e^{\pm \pi \eta_j \sqrt{-1}/\alpha_j}$
  where some odd integer $\eta_j$, 
  which implies that the order of $\rho(\ell)$ is even.
\end{proof}

\begin{remark}
  For every $j=0, 1, \ldots, m$,
  it holds that $\rho(\ell_j)^{\alpha_j} = -\I$ and $\rho(\ell_j)^{2\alpha_j} = \I$.
  The order of $\rho(\ell_j)$ must be less than or equal to $2 \alpha_j$.
  However the order of $\rho(\ell_0)$ is always $2$.
\end{remark}

We turn to the higher dimensional Reidemeister torsion of $M = \Sigma_* \times S^1$ for $\rho_{2N}$.
We have the following explicit values under the assumption that $\rho(h)=-\I$.
\begin{proposition}[Proposition~3.1 in~\cite{Kitano:RtorsionSeifertSLn}]
  \label{prop:log_tor_mfd}
  The Reidemeister torsion $\Tor{\mfd}{\rho_{2N}}$ is given by
  $2^{-2N(1-2g-m)}$, \ie
  $\log|\Tor{\mfd}{\rho_{2N}}| = -2N(1-2g-m) \log2$.
\end{proposition}

Now we are in position to apply our surgery formula (Theorem~\ref{thm:surgery_formula}) to
a Seifert fibered space $\Seifert$.
From Proposition~\ref{prop:log_tor_mfd} and Theorem~\ref{thm:surgery_formula},
we can derive the asymptotic behavior of the higher dimensional Reidemeister torsion 
for $\Seifert$.

\begin{theorem}
  \label{thm:asymptotics_Seifert}
  Let $\rho$ be an $\SL$-representation of
  $\Seifert$ such that
  $\rho(h) = -\I$.
  Then we can express the asymptotics of
  $\log |\Tor{\Seifert}{\rho_{2N}}|$ as follows:
  \begin{enumerate}
  \item
    $\displaystyle{
    \lim_{N \to \infty}
    \frac{
      \log |\Tor{\Seifert}{\rho_{2N}}|
    }{
      (2N)^2
    }
    =0
  }$,
  \item
    \label{item:leading_term_Seifert}
    $\displaystyle{
    \lim_{N \to \infty}
    \frac{
      \log |\Tor{\Seifert}{\rho_{2N}}|
    }{
      2N
    }
    = - \Big(2-2g - \sum_{j=1}^m \frac{\lambda_j - 1}{\lambda_j}\Big) \log 2
  }$
  \end{enumerate}
  where $2\lambda_j$ is the order of $\rho(\ell_j)$.
  
  In particular, if $\lambda_j$ is equal to $\alpha_j$ for all $j$,  then
  we have
  \begin{align*}
  \lim_{N \to \infty}
  \frac{
    \log |\Tor{\Seifert}{\rho_{2N}}|
  }{
    2N
  }
  &= -\Big(2-2g - \sum_{j=1}^m \frac{\alpha_j -1}{\alpha_j}\Big) \log 2\\
  &= -\chi \log 2
  \end{align*}
  where $\chi$ is the Euler characteristic of the base orbifold.
\end{theorem}

\begin{proof}[Proof of~Theorem~\ref{thm:asymptotics_Seifert}]  
  Applying Theorem~\ref{thm:surgery_formula}, we obtain
  $$
  \lim_{N \to \infty}
  \frac{
    \log |\Tor{\Seifert}{\rho_{2N}}|
  }{
    (2N)^2
  }
  =
  \lim_{N \to \infty}
  \frac{
    \log |\Tor{\mfd}{\rho_{2N}}|
  }{
    (2N)^2
  }
  = 0
  $$
  by Proposition~\ref{prop:log_tor_mfd}.
  Also it follows that
  \begin{align*}
  \lim_{N \to \infty}
  \frac{
    \log |\Tor{\Seifert}{\rho_{2N}}|
  }{
    2N
  }
  &=
  \lim_{N \to \infty}
  \frac{
    \log |\Tor{\mfd}{\rho_{2N}}|
  }{
    2N
  }
  - \Big(\sum_{j=0}^m \frac{1}{\lambda_j}\Big) \log 2\\
  &=
  -(1-2g-m) - \Big( 1 + \sum_{j=1}^{m} \frac{1}{\lambda_j}\Big) \log 2\\
  &=
  - \Big(2-2g - \sum_{j=1}^{m} \frac{\lambda_j - 1}{\lambda_j} \Big) \log 2.
  \end{align*}
\end{proof}

\begin{remark}
  \label{remark:special_case_lambda}
  It follows from the proof of Lemma~\ref{lemma:order_ell} that
  $\rho(\ell_j)^{2\alpha_j} = \I$ for all $j$.
  Each $\lambda_j$ in Theorem~\ref{thm:asymptotics_Seifert} is a divisor of 
  the corresponding $\alpha_j$.
\end{remark}

\begin{corollary}
  \label{cor:maximal_leading_term}
  The value $- \chi \log 2$ is the maximum in the limits~\eqref{item:leading_term_Seifert}
  of Theorem~\ref{thm:asymptotics_Seifert} for all $\SL$-representations
  sending $h$ to $-\I$.
\end{corollary}
\begin{proof}
  We can rewrite \eqref{item:leading_term_Seifert}
  in Theorem~\ref{thm:asymptotics_Seifert} as
  \begin{align}
    \lim_{N \to \infty}
    \frac{
      \log | \Tor{\SeifertZHS}{\rho_{2N}} |
    }{
      2N
    }
    &= - \Big(2-\sum_{j=1}^m \frac{\lambda_j -1}{\lambda_j}\Big) \log 2\notag\\
    &= - \chi \log 2 - \log 2\sum_{j=1}^m \Big(\frac{1}{\lambda_j} - \frac{1}{\alpha_j}\Big)
    \label{eqn:another_form_leading_term}.
  \end{align}
  Our claim follows from that each $\lambda_j$ is a divisor of $\alpha_j$ for $j=1, \ldots, m$.
\end{proof}

We also give the explicit form of the higher dimensional Reidemeister torsion 
for a Seifert fibered space. The following is the direct application of 
Lemma~\ref{lemma:MultLemma} (Multiplicativity Lemma).
\begin{proposition}
  Let $\rho$ be an $\SL$-representation of $\pi_1(\Seifert)$
  such that $\rho(h)=-\I$.
  Then we can express $\Tor{\Seifert}{\rho_{2N}}$ as
  $$
  \Tor{\Seifert}{\rho_{2N}}
  =
  2^{-2N(2-2g-m)}
  \cdot \prod_{j=1}^m \prod_{k=1}^N \Big(2 \sin \frac{\pi(2k-1)\eta_j}{2\alpha_j}\Big)^{-2}
  $$
  where $e^{\pm \pi \eta_j \sqrt{-1} / \alpha_j}$ are the eigenvalues of $\rho(\ell_j)$.
\end{proposition}

For the Reidemeister torsion of Seifert fibered spaces ($g>1$) with more general irreducible
$\SL[n]$-representations, we refer to \cite{Kitano:RtorsionSeifertSLn}.

\begin{remark}
  We do not require the irreducibility of $\rho_{2N} = \sigma_{2N} \circ \rho$.
  However our assumption that $\rho(h)=-\I$ guarantees the acyclicity of $\rho_{2N}$
  for all $N$.
\end{remark}

\subsection{The leading coefficients and the $\SU$-character varieties for Seifert fibered homology spheres}
\label{subsec:leadingCoeff_Seifert}
We have shown the explicit limits of the leading coefficients in 
the higher dimensional Reidemeister torsions for 
Seifert fibered spaces $\Seifert$.
The limit of the leading coefficient depends only on the order of $\rho(\ell_j)$ for 
an $\SL$-representation $\rho$ of $\pi_1(\Seifert)$.
In $\pi_1(\Seifert)$, 
we have the relations that $\ell_j = q_j^{\mu_j} h^{\nu_j}$ and 
$q_j^{\alpha_j} h^{\beta_j} = 1$ where 
$\alpha_j \nu_j - \beta_j \nu_j = -1$.
Under the assumption that $\rho(h) = -\I$, 
the order of $\rho(\ell_j)$ is determined by the order 
$\rho(q_j)$, \ie
the eigenvalues of $\rho(q_j)$.

Here and subsequently, 
following
the previous studies~\cite{FintushelStern:InstantonSeifert, KirkKlassen:RepresentationSeifert, BauerOkonek},
we assume that $\rho$ is irreducible.
Irreducible representations with the same eigenvalues for the generators $q_1, \ldots, q_m$ 
form a set with a structure of variety. 
When we also consider their conjugacy classes, it is known that the set of conjugacy classes 
also has a structure of variety.

We focus on Seifert fibered homology spheres and 
$\SU$-representations of their fundamental groups.
It has shown by~\cite{FintushelStern:InstantonSeifert, KirkKlassen:RepresentationSeifert, BauerOkonek}
that 
the set of conjugacy classes of irreducible $\SU$-representations 
for a Seifert fibered homology sphere
can be regarded as the set of smooth manifolds with even dimensions.
In the remain of paper, we deal with Seifert fibered homology spheres $\Seifert$ with $b=0$.
We will denote it briefly by $\SeifertZHS$ and write the set of conjugacy classes of 
irreducible $\SU$-representations of $\pi_1(\SeifertZHS)$ as
$$
\CharVar{\SeifertZHS}
= \mathrm{Hom}^{\mathrm{irr}}(\pi_1(\SeifertZHS), \SU) \,\big/\, \mathrm{conj.}
$$
which is called {\it the $\SU$-character variety}.

Each component in $\CharVar{\SeifertZHS}$ determined by 
the set of eigenvalues for $\SU$-elements corresponding to $q_1, \ldots, q_m$.
By the relation that $q_j^{\alpha_j} h^{\beta_j}=1$, 
the eigenvalues of $\rho(q_j)$ for an irreducible $\SU$-representation $\rho$ are given by
$
e^{\pm \pi \xi_j \sqrt{-1} / \alpha_j} 
$
($0 \leq \xi_j \leq \alpha_j$).
We will use the $m$-tuple $(\xi_1, \ldots, \xi_m)$ 
to denote the corresponding component in the $\SU$-character variety $\CharVar{\SeifertZHS}$
(for details, we refer to~\cite{FintushelStern:InstantonSeifert, KirkKlassen:RepresentationSeifert, BauerOkonek}).

\begin{proposition}[\cite{FintushelStern:InstantonSeifert, KirkKlassen:RepresentationSeifert, BauerOkonek}]
  \label{prop:Fact_CharVar}
  Let $\rho$ be an irreducible $\SU$-representation of $\SeifertZHS$.
  Suppose that the conjugacy class of $\rho$ is contained in a component $(\xi_1, \ldots, \xi_m)$.
  Then the dimension of $(\xi_1, \ldots, \xi_m)$ is equal to $2(n-3)$
  where $n$ is the number of $\xi_j$ such that $\xi_j \not = 0$, $\alpha_j$,
  \ie $\rho(q_j) \not = \pm \I$,
  in $j=1, \ldots, m$.
\end{proposition}

We will find components of $\CharVar{\SeifertZHS}$ containing 
the conjugacy class of an irreducible $\SU$-representation $\rho$ 
which makes 
the leading coefficient in the logarithm of $|\Tor{\SeifertZHS}{\rho_{2N}}|$ 
converge to $-\chi \log 2$.
\begin{proposition}
  \label{prop:leading_term_char_var}
  Let $\rho$ be an irreducible $\SU$-representation such that
  $\rho(h) = -\I$.
  The leading coefficient of $\log|\Tor{\SeifertZHS}{\rho_{2N}}|$
  converges to $-\chi \log 2$ if and only if
  the conjugacy class of $\rho$ is contained in 
  a $2(m-3)$-dimensional component $(\xi_1, \ldots, \xi_m)$ of $\CharVar{\SeifertZHS}$
  such that $\alpha_j$ and $\xi_j$ are coprime for all $j$.
\end{proposition}
Before proving Proposition~\ref{prop:leading_term_char_var}, 
let us observe a relation between the order of $\rho(\ell_j)$ and $\xi_j$.
\begin{lemma}
  \label{lemma:xi_alpha}
  Suppose that an irreducible $\SU$-representation $\rho$ satisfies that
  $\rho(h) = -\I$ and
  its conjugacy class is contained in a component $(\xi_1, \ldots, \xi_m)$.
  Then the order of $\rho(\ell_j)$ is equal to 2$\alpha_j$,
  \ie $\lambda_j = \alpha_j$,
  if and only if $\xi_j$ and $\alpha_j$ are coprime.
\end{lemma}
\begin{proof}
  We can express $\ell_j$ as $\ell_j = q_j^{\mu_j} h^{\nu_j}$ where
  $\alpha_j \nu_j - \beta_j \mu_j = -1$.
  Since the eigenvalues of $\rho(q_j)$ are given by $e^{\pm \pi \xi_j \sqrt{-1}/\alpha_j}$, 
  we can diagonalize $\rho(\ell_j)$ as
  $$
  \rho(\ell_j) \sim
  \begin{pmatrix}
    e^{ \pi \eta_j \sqrt{-1}/\alpha_j} & 0 \\
    0 & e^{ -\pi \eta_j \sqrt{-1}/\alpha_j}
  \end{pmatrix}
  \quad (\eta_j = \mu_j \xi_j - \alpha_j \nu_j).
  $$
  Suppose that $\xi_j$ and $\alpha_j$ are coprime.
  It follows from $(\alpha_j, \mu_j) = 1$ that
  the g.c.d. $(\alpha_j, \eta_j)$ coincides with $(\alpha, \xi_j) = 1$.
  Since the order of $\rho(\ell_j)$ is $2\lambda_j$, we can see that $\rho(\ell_j)^{\lambda_j} = -\I$.
  Hence $\alpha_j$ divides $\lambda_j$, which implies that $\lambda_j = \alpha_j$
  from that $\lambda_j$ is a divisor of $\alpha_j$.
  Similarly, when the order of $\rho(\ell_j)$ is $2\alpha_j$, the g.c.d. $(\alpha_j, \xi_j)$ must be $1$.
\end{proof}

\begin{remark}
  Under the assumption that $\rho(h)=-\I$,
  if $\xi_j$ is equal to $0$ or $\alpha_j$, then $\rho(\ell_j) = -\I$.
  Hence $\lambda_j = 1$.
\end{remark}

\begin{proof}[Proof of Proposition~\ref{prop:leading_term_char_var}]
  According to Eq.~\eqref{eqn:another_form_leading_term},
  every $\lambda_j$ coincides with $\alpha_j$ for all $j$
  if and only if 
  the leading coefficient of $\log|\Tor{\SeifertZHS}{\rho_{2N}}|$ converges to $-\chi \log 2$.
  By Lemma~\ref{lemma:xi_alpha},
  we can rephrase $\lambda_j = \alpha_j$ for all $j$ as $[\rho] \in (\xi_1, \ldots, \xi_m)$
  with the g.c.d. $(\alpha_j, \xi_j)=1$ for all $j$.
  In particular, 
  it is seen from Proposition~\ref{prop:Fact_CharVar} that the dimension of $(\xi_1, \ldots, \xi_m)$
  is equal to $2(m-3)$.
\end{proof}

In special cases that every $\alpha_j$ is prime in the Seifert index of $\SeifertZHS$,
we obtain a simple correspondence between the limits of $\log|\Tor{\SeifertZHS}{\rho_{2N}}| / (2N)$
and components of $\CharVar{\SeifertZHS}$.
\begin{theorem}
  \label{thm:max_min_leading_coeff}
  Let $\rho$ be an irreducible $\SU$-representation of $\pi_1(\SeifertZHS)$
  such that $\rho(h)= -\I$.
  Suppose that every $\alpha_j$ is prime and $\alpha_1 < \cdots < \alpha_m$.
  If the conjugacy class of $\rho$ is contained in a component $(\xi_1, \ldots, \xi_m)$, then
  the leading coefficient of $\log|\Tor{\SeifertZHS}{\rho_{2N}}|$ converges to
  \begin{equation}
    \label{eqn:simple_form_leading_coeff}
  \lim_{N \to \infty}
  \frac{
    \log|\Tor{\SeifertZHS}{\rho_{2N}}|
  }{
    2N
  }
  = - \Big(2-\sum_{\xi_j \not = 0, \alpha_j} \frac{\alpha_j -1}{\alpha_j}\Big)\log 2.
  \end{equation}
  If the set $\{ [\rho] \in \CharVar{\SeifertZHS} \,|\, \rho(h) = -\I \}$ has $0$ and $2(m-3)$-dimensional
  components, then the limit takes the maximum $-\chi \log 2$ on only all top-dimensional components and
  takes the minimum on some $0$-dimensional components.
\end{theorem}
\begin{proof}
  As seen in Remark~\ref{remark:special_case_lambda},
  we have $\lambda_j =1$ if $\xi_j = 0$ or $\alpha_j$.
  Substituting $\lambda_j = 1$ into the equality above Eq.~\eqref{eqn:another_form_leading_term}
  for the corresponding index $j$, 
  we have the limit~\eqref{eqn:simple_form_leading_coeff}.
  From Proposition~\ref{prop:leading_term_char_var}
  and our assumption, it follows  
  that the limit takes the maximum $-\chi \log 2$ on all top-dimensional components.

  It remains to prove that the limit takes the minimum on a $0$-dimensional component.
  Since the limit is expressed as Eq.~\eqref{eqn:simple_form_leading_coeff},
  we consider the minimum of $\sum_{\xi_j \not = 0, \alpha_j} (\alpha_j -1)/\alpha_j$.
  Each $0$-dimensional component is given by $(\xi_1, \ldots, \xi_m)$ for all $\xi_j = 0, \alpha_j$ 
  except three $\xi_{j_1}$, $\xi_{j_2}$ and $\xi_{j_3}$.
  We need to consider two cases: (i) $\alpha_1=2$ and (ii) $\alpha_1 \geq 3$.
  In the case that $\alpha_1 = 2$,
  $\xi_1$ must be $1$ for all components $(\xi_1, \ldots, \xi_m)$
  since we have $\rho(q_1)^{\alpha_1} = -\I$ from the assumption that $\rho(h)=-\I$.
  The sum $(\alpha_{j_1} -1) / \alpha_{j_1} + (\alpha_{j_2} -1) / \alpha_{j_2} + (\alpha_{j_3} -1) / \alpha_{j_3}$
  turns into 
  $$
  \frac{1}{2} +
  \frac{\alpha_{j_2} -1}{\alpha_{j_2}} +
  \frac{\alpha_{j_3} -1}{\alpha_{j_3}}
  < \frac{5}{2}.$$
  On the other hand, it is easily seen that for higher dimensional components,
  $$
  \frac{1}{2} +
  \frac{\alpha_{i_2} -1}{\alpha_{i_2}} +
  \frac{\alpha_{i_3} -1}{\alpha_{i_3}} +
  \frac{\alpha_{i_4} -1}{\alpha_{i_4}} +\cdots 
  \geq \frac{1}{2} + \frac{3 - 1}{3} + \frac{5-1}{5} + \frac{7-1}{7}
  > \frac{5}{2}.$$
  Hence the minimum lies in a $0$-dimensional component.
  
  In the other case that $\alpha_1 \geq 3$, it is clear that 
  $$
  \frac{\alpha_{j_1} -1}{\alpha_{j_1}} +
  \frac{\alpha_{j_2} -1}{\alpha_{j_2}} +
  \frac{\alpha_{j_3} -1}{\alpha_{j_3}}
  < 3.
  $$
  On the other hand,
  we can see that
  $$
  \frac{\alpha_{i_1} -1}{\alpha_{i_1}} +
  \frac{\alpha_{i_2} -1}{\alpha_{i_2}} +
  \frac{\alpha_{i_3} -1}{\alpha_{i_3}} +
  \frac{\alpha_{i_4} -1}{\alpha_{i_4}} +\cdots
  \geq
  \frac{3-1}{3} + \frac{5-1}{5} + \frac{7-1}{7} + \frac{11-1}{11}
  > 3.
  $$
  The minimum of the limits lies on $0$-dimensional components.
\end{proof}

\subsection{Examples for Seifert fibered homology spheres}
\label{subsec:examples}
We will see two examples of Theorem~\ref{thm:max_min_leading_coeff} 
and an example which shows that  $\log |\Tor{\SeifertZHS}{\rho_{2N}}| / (2N)$
does not converges to the maximum $-\chi \log 2$
on all top--dimensional components
if some $\alpha_j$ is not prime and $(\alpha_j, \xi_j) \not = 1$.

\subsubsection{$M\hbox{$\big(\frac{2}{\beta_1}, \frac{3}{\beta_2}, \frac{7}{\beta_3}\big)$}$}
We can choose that $\beta_1=1$ and $\beta_2 = \beta_3 = -1$ by 
the requirement that
$2 \cdot 3 \cdot 7 (\beta_1 / 2 + \beta_2 / 3 + \beta_3 / 7) =1$.
The Brieskorn homology $3$-sphere $\ExI$ also corresponds to the surgery along 
$(2, 3)$-torus knot with slope $1$ in Subsection~\ref{subsec:example_torusknots}.
From the presentation:
$$
\pi_1(\ExI)
=\langle
q_1, q_2, q_3, h \,|\,
[q_j, h]=1, q_j^{\alpha_j} h^{\beta_j} = 1, q_1 q_2 q_3 =1
\rangle,
$$
every irreducible $\SU$-representation sends $h$ to $-\I$ and 
the $\SU$-character variety of $\ExI$ consists of
$(\xi_1, \xi_2, \xi_3) = (1, 1, 3)$ and $(1, 1, 5)$.
For details about the computation of $\SU$-character varieties,
we refer to \cite{FintushelStern:InstantonSeifert, KirkKlassen:RepresentationSeifert}
and~\cite[Lecture~$14$]{Saveliev99:LectureTopology3mfd}.

Let $\rho$ be an irreducible $\SU$-representation of $\ExI$.
By the relations that $\ell_j = q_j^{\mu_j} h^{\nu_j}$ and $\alpha_j \nu_j - \beta_j \mu_j = -1$
($0 < \mu_j < \alpha_j$),
we obtain that 
$$
\rho(\ell_1) = \rho(q_1),\,
\rho(\ell_2)= - \rho(q_2)^2,\,
\rho(\ell_3) = - \rho(q_3)^6.
$$
From Eqs.~$\rho(q_1)^2 = -\I$, $\rho(q_2)^3 = -\I$ and $\rho(q_3)^7 = -\I$,
the orders $2\lambda_j$ of $\rho(\ell_j)$ are given by 
$$
2\lambda_1 = 4,\,
2\lambda_2 = 6,\,
2\lambda_3 = 14.
$$
By Theorem~\ref{thm:asymptotics_Seifert},
for the both cases of $[\rho] \in (1, 1, 3)$ and $[\rho] \in (1, 1, 5)$,
we can see that 
$$
\lim_{N \to \infty}
\frac{
\log |\Tor{\ExI}{\rho_{2N}}|}{2N}
= \Big(1 - \frac{1}{2} - \frac{1}{3} -\frac{1}{7}\Big) \log2.
$$
The limit of the leading coefficient takes the maximum $-\chi \log 2$ on all top--dimensional components
(see also Theorem~\ref{thm:asymptotics_Brieskorn}).

\subsubsection{$M\hbox{$\big(\frac{2}{\beta_1}, \frac{3}{\beta_2}, \frac{5}{\beta_3}, \frac{7}{\beta_4}\big)$}$}
Let us choose $\beta_1=1$, $\beta_2 = \beta_3 = -2$ and $\beta_4=4$.
The subvariety $\{[\rho] \in \CharVar{\ExII} \,|\, \rho(h) = -\I\}$
consists of eight points and six $2$-dimensional spheres.
(For details, see~~\cite[Lecture~$14$]{Saveliev99:LectureTopology3mfd})

Each $0$-dimensional component corresponds to the parameter $(\xi_1, \xi_2, \xi_3, \xi_4)$ in:
\begin{equation}
  \label{eqn:2dim_compo}
\left\{
\begin{array}{c}
  (1, 0, 2, 2),\, (1, 0, 2, 4),\, (1, 0, 2, 6),\, (1, 0, 4, 4), \\
  (1, 2, 0, 2),\, (1, 2, 0, 4),\, (1, 2, 2, 0),\, (1, 2, 4, 0)
\end{array}
\right\}
\end{equation}
and each $2$-dimensional components are given by $(\xi_1, \xi_2, \xi_3, \xi_4)$ in 
$$
\left\{
\begin{array}{cc}
  (1, 2, 2, 2),\, (1, 2, 2, 4),\, (1, 2, 2, 6), \\
  (1, 2, 4, 2),\, (1, 2, 4, 4),\, (1, 2, 4, 6)
\end{array}
\right\}.
$$
We can express $\ell_j$ ($1 \leq j \leq 4$) as
$$
\ell_1 = q_1,\,
\ell_2 = q_2 h^{-1},\,
\ell_3 = q_3^2 h^{-1},\,
\ell_4 = q_4 h.  
$$

Let $\rho$ be an irreducible $\SU$-representation of $\pi_1(\ExII)$ such that $\rho(h) = -\I$.
We have the following table between the $0$-dimensional components and the orders of $\rho(\ell_j)$
for $j=1, 2, 3, 4$:
$$
\begin{tabular}{c|c}
$(\xi_1, \xi_2, \xi_3, \xi_4)$ & $\lambda_j$: the half of the order of $\rho(\ell_j)$ \\
\hline
$(\xi_1, 0, \xi_3, \xi_4)$ & $\lambda_1 = 2$, $\lambda_2 = 1$, $\lambda_3 = 5$, $\lambda_4 = 7$ \\
$(\xi_1, \xi_2, 0, \xi_4)$ & $\lambda_1 = 2$, $\lambda_2 = 3$, $\lambda_3 = 1$, $\lambda_4 = 7$ \\
$(\xi_1, \xi_2, \xi_3, 0)$ & $\lambda_1 = 2$, $\lambda_2 = 3$, $\lambda_3 = 5$, $\lambda_4 = 1$.
\end{tabular}
$$
Hence for $[\rho] \in (\xi_1, \xi_2, \xi_3, \xi_4)$\, $(\xi_2 = 0)$,
by Theorem~\ref{thm:asymptotics_Seifert},
we obtain  
\begin{align*}
  \lim_{N \to \infty}
  \frac{
    \log |\Tor{\ExII}{\rho_{2N}}|}{2N}
  &= 
  -\Big(2 - \sum_{j \not = 2} \frac{\lambda_j -1}{\lambda_j}\Big)\log 2\\
  &= -\chi \log 2 - \frac{2}{3} \log 2,
\end{align*}
for $[\rho] \in (\xi_1, \xi_2, \xi_3, \xi_4)$\, ($\xi_3 = 0$),
\begin{align*}
  \lim_{N \to \infty}
  \frac{
    \log |\Tor{\ExII}{\rho_{2N}}|}{2N}
  &= 
  -\Big(2 - \sum_{j \not = 3} \frac{\lambda_j -1}{\lambda_j}\Big)\log 2\\
  &= -\chi \log 2 - \frac{4}{5} \log 2
\end{align*}
and for $[\rho] \in (\xi_1, \xi_2, \xi_3, \xi_4)$\, ($\xi_4 = 0$),
\begin{align*}
  \lim_{N \to \infty}
  \frac{
    \log |\Tor{\ExII}{\rho_{2N}}|}{2N}
  &= 
  -\Big(2 - \sum_{j \not = 4} \frac{\lambda_j -1}{\lambda_j}\Big)\log 2\\
  &= -\chi \log 2 - \frac{6}{7} \log 2.
\end{align*}

When the conjugacy class $[\rho]$ is contained in the $2$-dimensional components in~\eqref{eqn:2dim_compo},
then it is seen that $\lambda_1 = 2$, $\lambda_2 = 3$, $\lambda_3 = 5$ and $\lambda_4 = 7$.
Hence we obtain the maximum of the limits:
$$
\lim_{N \to \infty}
\frac{
  \log |\Tor{\ExII}{\rho_{2N}}|}{2N}
= -\chi \log 2.
$$
The limit of $\log |\Tor{\ExII}{\rho_{2N}}| / (2N)$ takes the minimum at 
the components $(1, 2, 2, 0)$ and $(1, 2, 4, 0)$.

\subsubsection{$M\hbox{$\big(\frac{5}{\beta_1}, \frac{6}{\beta_2}, \frac{7}{\beta_3}\big)$}$}
Let us choose $\beta_1 = 3$, $\beta_2 = -1$ and $\beta_3 = -3$.
The subvariety $\{ [\rho] \in \CharVar{\ExIII} \,|\, \rho(h) = -\I\}$
consists of eight points
which are given by the following set of $(\xi_1, \xi_2, \xi_3)$:
$$
\left\{
\begin{array}{c}
  (1, 1, 1),\, (1, 3, 3),\, (1, 5, 5), \\
  (3, 1, 5),\, (3, 3, 1),\, (3, 3, 3),\, (3, 3, 5),\, (3, 5, 3)
\end{array}
\right\}
$$

We separate the above subvariety into two subset $X_1$ and $X_2$ as follows:
\begin{align*}
X_1 &:= \{ (1, 1, 1), (1, 5, 5), (3, 1, 5), (3, 5, 3) \}\\
X_2 &:= \{ (1, 3, 3), (3, 3, 1), (3, 3, 3), (3, 3, 5) \}
\end{align*}
where every component is $0$-dimensional.
Every components in $X_1$ satisfies that $(\alpha_j, \xi_j)=1$ for all $j$.
On the other hand, each component in $X_2$ satisfies that $(\alpha_1, \xi_1)=(\alpha_3, \xi_3)= 1$ and 
$(\alpha_2, \xi_2) = 3$.

In the case that the conjugacy class $[\rho]$ is contained in $X_1$,
we can also compute similarly $\lambda_1 = 5$, $\lambda_2 = 6$ and $\lambda_3 = 7$.
Therefore we obtain the following limit, by Theorem~\ref{thm:asymptotics_Seifert}
$$
\lim_{N \to \infty}
\frac{
  \log |\Tor{\ExIII}{\rho_{2N}}|
}{
  2N
}
= -\chi \log 2.
$$

In the case of the conjugacy class $[\rho]$ is contained in $X_2$,
the order of $\rho(\ell_2)$ is $4$ since $\ell_2 = q_2^{5} h^{-1}$ and $\rho(q_2)^2 = -\I$.
Hence $\lambda_2$ equals to $2$. 
Similar computations yield $\lambda_1 = 5$ and $\lambda_3 = 7$.
Furthermore Theorem~\ref{thm:asymptotics_Seifert} gives the following limit:
\begin{align*}
  \lim_{N \to \infty}
  \frac{
    \log |\Tor{\ExIII}{\rho_{2N}}|
  }{
    2N
  }
  &= -\Big(2 - \sum_{j=1}^3 \frac{\lambda_j - 1}{\lambda_j}\Big) \log 2\\
  &= -\chi \log 2 - \frac{1}{3} \log 2.
\end{align*}

Therefore
we have the top-dimensional components  
where the limit of the leading coefficient in 
$\log |\Tor{\ExIII}{\rho_{2N}}|$
does not take the maximum.

\section*{Acknowledgment}
The author wishes to express his thanks to Joan Porti, Takayuki Morifuji and Takahiro Kitayama
for helpful suggestions to start the computation in this manuscript.
This research was supported by 
Research Fellowships of the Japan Society for the Promotion of Science for Young Scientists.

%%%%%%%%%%%%%%%%%%%%%%%%%%%%%%%%%%%%%%%%%%%%%%%%%%%%%%%%%%%%%%%%%%%% 
% reference
%%%%%%%%%%%%%%%%%%%%%%%%%%%%%%%%%%%%%%%%%%%%%%%%%%%%%%%%%%%%%%%%%%%% 
\bibliographystyle{amsalpha}
\bibliography{asymptoticsTorsionSeifert}

\end{document}